\documentclass[12pt]{article} 
\usepackage[margin=1in]{geometry}
\usepackage{amsfonts,latexsym,amsthm,amssymb,amsmath,amscd,esint,dsfont}
\usepackage{graphicx}
\graphicspath{ {./figure/} }
\usepackage[margin=10pt,font=small,labelfont=bf]{caption}

\usepackage{cases}			
\usepackage{textcomp, gensymb}
\usepackage{xcolor}

\usepackage{tabularx}
\usepackage{array}

\makeatletter
\def\namedlabel#1#2{\begingroup
	\def\@currentlabel{#2}
	\label{#1}\endgroup
}
\makeatother

\usepackage[parfill]{parskip}
\usepackage{fancyhdr}

\newcommand{\HH}{\mathbb{H}}

\newcommand{\NN}{\mathbb{N}}

\newcommand{\RR}{\mathbb{R}}

\newcommand{\ZZ}{\mathbb{Z}}

\newcommand{\cB}{\mathcal{B}}

\newcommand{\cL}{\mathcal{L}}

\newcommand{\cT}{\mathcal{T}}

\newcommand{\pdv}[2]{\frac{\partial #1}{\partial #2}}

\newcommand{\Isom}{\operatorname{Isom}}

\newcommand{\sgn}{\operatorname{sgn}}

\newtheorem{theorem}{Theorem}[section]
\newtheorem{corollary}[theorem]{Corollary}
\newtheorem{lemma}[theorem]{Lemma}

\newtheorem{proposition}[theorem]{Proposition}

\theoremstyle{definition}
\newtheorem{definition}[theorem]{Definition}

\newtheorem{remark}[theorem]{Remark}

\usepackage{hyperref}
\hypersetup{
    colorlinks=true,
    linkcolor=blue,
    filecolor=magenta,      
    urlcolor=cyan,
    pdftitle={Overleaf Example},
    pdfpagemode=FullScreen,
}

\usepackage[
backend=biber,
style=alphabetic,
sorting=ynt
]{biblatex}
\title{Vertex-minimal hyperbolic origami 2-torus}
\author{Zhengyu Zou}
\date{\today}

\begin{document}
\maketitle

\begin{abstract}
    We show that there exists a geodesic triangulation $\cT$ of a hyperbolic genus 2 surface $\Sigma_2$ with 10 vertices and an isometric polyhedral embedding $S: \Sigma_2 \hookrightarrow \HH^3$ that sends the triangles in $\cT$ to geodesic triangles in $\HH^3$. We call this type of embedding a \emph{hyperbolic origami 2-torus}. Since 10 is the combinatorially minimum number of vertices required to triangulate a genus 2 surface, this paper settles the question of minimum number of vertices required to obtain a hyperbolic origami 2-torus. 
\end{abstract}

\section{Introduction} \label{sec:intro}

\subsection{Context and Problem Statement}

A \emph{hyperbolic 2-torus} $\Sigma_2$ is a quotient of $\HH^2$ by the action of a surface group $\Gamma < \Isom^+(\HH^2)$ such that $\HH^2 / \Gamma$ is a genus 2 surface. Here, $\Gamma$ is isomorphic to the group with presentation $\langle a, b, c, d \, | \, [a,b][c,d] \rangle$, where $[a,b] = aba^{-1}b^{-1}$ denotes the commutator of the two elements, and the action of $\Gamma$ on $\HH^2$ is free and properly discontinuous. Every hyperbolic 2-torus $\Sigma_2$ inherits the structure of a Riemannian manifold of constant sectional curvature -1 from $\HH^2$. A \emph{hyperbolic origami 2-torus} is a map $S: (\Sigma_2, \cT) \hookrightarrow \HH^3$, where $\cT$ is a geodesic triangulation of $\Sigma_2$, and $S$ restricts to isometries on the triangles of $\cT$. Concretely, we think of a hyperbolic origami 2-torus $S$ as gluing hyperbolic triangles together in $\HH^3$, where the cone angles at each vertex are $2\pi$.

The study of graph embeddings and triangulations of closed surfaces can be dated back to the late 19th century \cite{Heawood1890}. The minimum number of vertices $n_0$ to triangulate a closed surface of Euler characteristic $e$ is given by 
\begin{equation} \label{eqn:minimum vertex to triangulate surface}
    n_0 \geq \bigg\lfloor \frac{7 + \sqrt{49 - 24 e}}{2} \bigg\rfloor 
\end{equation}
Equation \eqref{eqn:minimum vertex to triangulate surface} is proved by Ringel \cite{Ringel1955} for the non-orientable cases and later by Jungerman and Ringel \cite{JR1980} for the orientable cases. However, the bound in Equation \eqref{eqn:minimum vertex to triangulate surface} is not tight for the 2-torus. In \cite{HUNEKE1978258}, Huneke showed that the minimum number of vertices required to triangulate a 2-torus is 10, in contrast with a looser lower bound of 8 from Equation \eqref{eqn:minimum vertex to triangulate surface}. In 2008, Lutz \cite{Lutz2008} enumerated all 865 isomorphism classes of 10-vertex triangulations of the 2-torus. 

Geometers were interested in whether these triangulated surfaces admit piecewise-affine embeddings into $\RR^3$ that are affine on the triangle pieces \cite{bokowski1987new, brehm1987maximally, BokowskiBrehm1989, burago1995isometric, HougardyLutzZelke2006}. These embeddings are called \emph{polyhedral embeddings}. A polyhedral embedding of a 10-vertex triangulation of the 2-torus was first discovered by Brehm \cite{Brehm1981}. Later, Hougardy, Lutz, and Zelke \cite{HougardyLutzZelke2007} developed a computer algorithm to enumerate polyhedral embeddings of all 10-vertex triangulations of the 2-torus. 
However, a fundamental problem remains unsolved: if we take into account the hyperbolic structures on a 2-torus, these polyhedral embeddings fail to preserve them. First, it's impossible to map hyperbolic triangles isometrically to geodesic triangles in $\RR^3$ due to the Gauss--Bonnet theorem. Thus, it is more natural to consider embeddings of $(\Sigma_2, \cT)$ into $\HH^3$, where the images of triangles in $\cT$ are isometric geodesic triangles in the hyperbolic 3-space. 

There has been significant progress on polyhedral embeddings of the flat torus. See \cite{Zalgaller2000, BernHayes2008,segerman2016visualizing,LazarusTallerie2024, tsuboi2024origami}. A \emph{flat polyhedral torus} is a piecewise-linear embedding of a triangulated torus $T^2$ into $\RR^3$ that restricts to linear isometries on the triangles, where the triangles in $T^2$ come from a triangular tiling of the universal cover $\RR^2$. A very recent work of Schwartz \cite{schwartz2025vertexminimalpapertori} showed that a flat polyhedral torus of 7 vertices doesn't exist, but one can obtain one with 8 vertices with 2-fold symmetry. This settled the question of the minimum number of vertices required for a flat polyhedral torus. 
In this paper, we prove that the minimum number of vertices required for a hyperbolic origami 2-torus is 10. 

\subsection{The Candidate Model 2-Torus}
We state the main result of this paper.

\begin{theorem} \label{thm:flat 2-torus for all vertices}
    There exists a hyperbolic 2-torus $\Sigma_2$ along with a 10-vertex geodesic triangulation $\cT$ and a hyperbolic origami 2-torus $S: (\Sigma_2, \cT) \hookrightarrow \HH^3$.
\end{theorem}

\begin{figure}[ht]
    \centering
    \includegraphics[width=0.3\columnwidth]{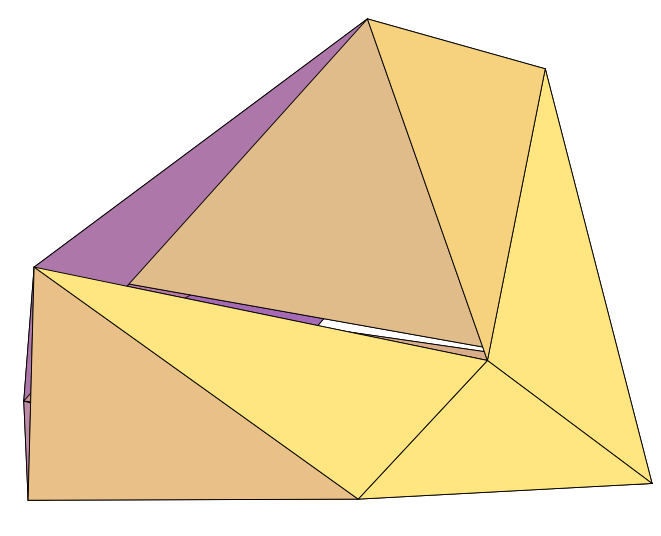}
    \hspace{20pt}
    \includegraphics[width=0.3\columnwidth]{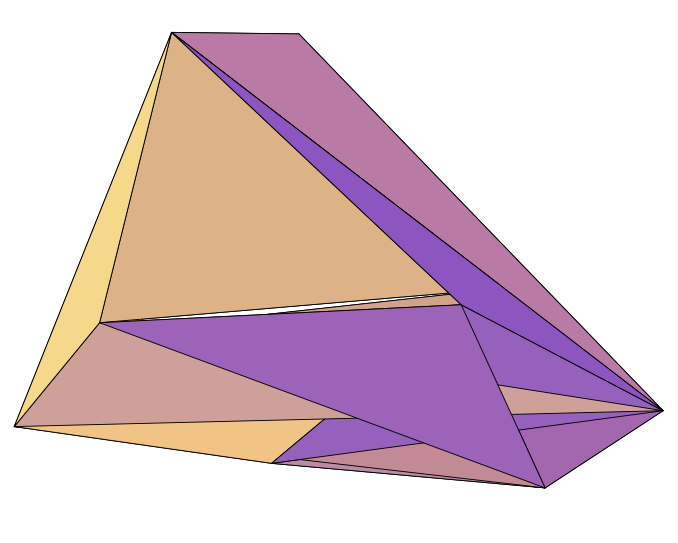}
    \caption{The candidate 2-torus $\hat S$ in the Beltrami--Klein model $\HH^3$, shown from opposite sides of a ``hole'' passing through it. Triangles are shaded along a yellow--purple spectrum according to the $y$-coordinates of their centroids, ordered from highest to lowest.}
    \label{fig:surface}
\end{figure}

\begin{figure}[ht]
    \centering
    
    \def\svgwidth{0.85\columnwidth}
    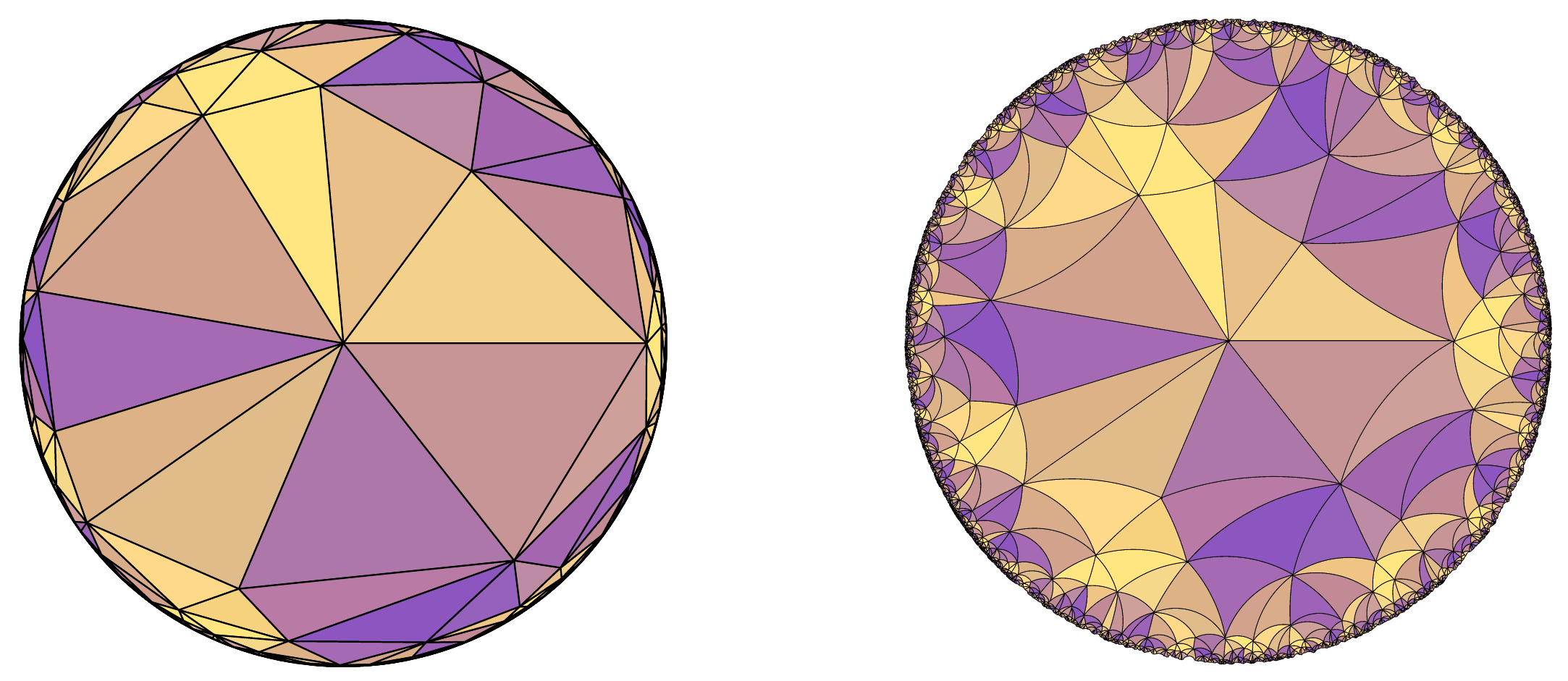

    \caption{Left: The induced triangle tiling from $\cT$ in the Beltrami--Klein model of $\HH^2$. Right: The tiling in the Poincar\'e disk model. The colors of the triangles are derived from Figure \ref{fig:surface}.}
    \label{fig:univ cover}
\end{figure}

The information of a hyperbolic origami 2-torus is captured entirely by the triangulation $\cT$ and the image of the vertices of $\cT$ under the map $S: (\Sigma_2, \cT) \rightarrow \HH^3$, as long as the cone angles of the vertices of $\cT$ are all $2\pi$. We remark that Theorem \ref{thm:flat 2-torus for all vertices} implies the existence of hyperbolic origami 2-tori with $n$ vertices for all $n \geq 10$; one simply uses the 10-vertex embedding and adds vertices to the interior of the triangles in $\cT$ to obtain triangulations with more vertices. 

Consider the following 10-vertex triangulation $\cT = (V, E, F)$ with vertices in $V$, edges in $E$, and oriented triangles in $F$. We index $V$ by $\{0, \ldots, 9\}$. 
We list the degree sequence of the 10 vertices in $V$:
\begin{equation*}
    (d_0, d_1, d_2, d_3, d_4, d_5, d_6, d_7, d_8, d_9)
    = 
    (9, \, 8, \, 8, \, 8, \, 8, \, 7, \, 7, \, 6, \, 6, \, 5) .
\end{equation*}

We also list out the indices of the 24 oriented triangles in $F$ in the form of 3-tuples. A tuple $(i, j, k)$ indicates an oriented face joined by vertices $i, j, k$. 
\begin{equation} \label{eqn:triangles in the triangulation}
    \begin{matrix}
        (0, 1, 7) & (0, 2, 1) & (0, 3, 6) & (0, 4, 9) & (0, 5, 3) & (0, 6, 4) \\
        (0, 7, 8) & (0, 8, 5) & (0, 9, 2) & (1, 2, 4) & (1, 3, 7) & (1, 4, 6) \\
        (1, 5, 8) & (1, 6, 5) & (1, 8, 3) & (2, 3, 8) & (2, 6, 3) & (2, 7, 4) \\
        (2, 8, 7) & (2, 9, 6) & (3, 4, 7) & (3, 5, 4) & (4, 5, 9) & (5, 6, 9)
    \end{matrix}
\end{equation}

The proof of Theorem \ref{thm:flat 2-torus for all vertices} consists of two steps. First, we find a polyhedral 2-torus $\hat S$ that is ``robustly embedded'' and ``almost isometric.'' Next, we argue that an isometric one exists by perturbing the coordinates of $\hat S$. Figure \ref{fig:surface} shows the polyhedral 2-torus $\hat S$ embedded in the Beltrami--Klein model of the hyperbolic 3-space (recall that in the Klein model, geodesics are straight lines and geodesic triangles are Euclidean triangles). In Figure \ref{fig:univ cover}, we draw the induced tiling of the universal cover $\HH^2$ by the ``almost geodesic'' triangulation on $\Sigma_2$ by ``almost isometrically'' lifting the triangles in $\hat S$ to $\HH^2$ and gluing them together. Since $\hat S$ is not a hyperbolic origami 2-torus, the picture of the tiling in the universal cover is only meant to be an approximation up to some error (about $10^{-28}$), but the error is too small to be discernible by human eyes so that visually the picture looks reasonable. See Figure \ref{fig:slice_xy}, \ref{fig:slice_xz}, \ref{fig:slice_yz} for the intersection of $\hat S$ with the $xy$-plane, $xz$-plane, and $yz$-plane. 

\begin{figure}[ht]
    \centering
    
    \def\svgwidth{0.4\columnwidth}
    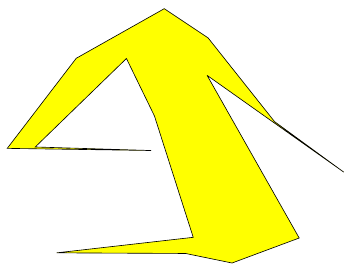

    \caption{The intersection of $\hat S$ with the $xy$-plane, which is the boundary of a simply-connected component colored in yellow.}
    \label{fig:slice_xy}
\end{figure}

Here we give the exact coordinates of the vertices $\hat X_0, \ldots, \hat X_9$ of our candidate 2-torus $\hat S$ in the Beltrami--Klein model:
\begin{equation} \label{eqn:coordinates of S hat}
    \begin{matrix}
        0.7315 & 0.0202 & 0.28688022781563440615364787558404 \\
        -0.316 & 0.5792 & -0.2252919753182150895621576631383 \\
        0.3426 & -0.592 & -0.22851917827575874828458055465199 \\
        -0.4323 & -0.592 & -0.23272863894943839798113773793091 \\
        -0.7303 & 0.04 & -0.22959077803009316431117678104662 \\
        0.1464 & 0.6149 & 0.13588682780065943976868637469759 \\
        -0.5154 & 0.0395 & 0.46102777383206591809202059407883 \\
        0.6649 & -0.1156 & -0.22651115997910956793325851442687 \\
        0.152 & 0.2539 & -0.23985732806740791506457015735734 \\
        -0.03 & 0.0606 & 0.64396456614136038886542316026992
    \end{matrix}    
\end{equation}

\begin{figure}[ht]
    \centering
    
    \def\svgwidth{0.5\columnwidth}
    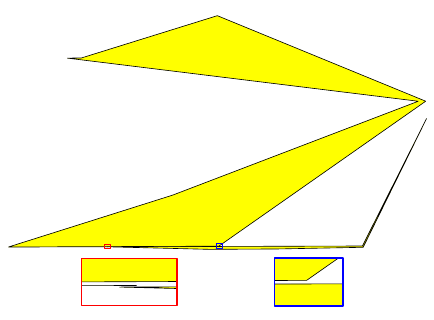

    \caption{The intersection of $\hat S$ with the plane $y = 0$, which is the boundary of two simply-connected components colored in yellow. The boxes contain the magnified view of their respective regions in the original slice.}
    \label{fig:slice_xz}
\end{figure}

We explain how we chose the triangulation $\cT$ and obtained these coordinates. We start with all triangulations $\cT_n = (V_n, E_n, F_n)$ and polyhedral embeddings $S_n: (\Sigma_2, \cT_n) \hookrightarrow \RR^3$ from \cite{HougardyLutzZelke2007} for all 865 triangulations $n = 1, \ldots, 865$, where $S_n$ maps vertices of $\cT_n$ to the integer lattice $\ZZ^3 \cap [0,4]^3$. First, we apply a translation and scale the vectors so that the center-of-mass of the vertices is at the origin, and that $\max_{X \in V_n} \| X \| = \frac{1}{2}$. This ensures that $S_n(\Sigma_2) \subset B_1(0)$, so we can treat $S_n$ as an embedding into the Klein model of $\HH^3$. Then, for each $n$, we compute the hyperbolic angles of all triangles in $S_n$ and all the cone angles. We found that $S_{265}$ has the minimal cone angle difference to $2\pi$ among all other embeddings. After that, we apply a hill-climbing algorithm by randomly perturbing the coordinates of $S_{265}$ to minimize the objective function 
\begin{equation*}
    \cL(S) = \max_{i = 0, \ldots, 9} | \theta_i(S) - 2 \pi |
\end{equation*}
where $\theta_i(S)$ is the cone angle of $S$ at vertex $i$. This gives us a surface $S'_{265}$ with $\cL(S'_{265}) \leq 10^{-14}$. Let $z = (z_0, \ldots, z_9)$ denote the vector of $z$-coordinates of $S$. We implement a high-precision version of Newton's method on the multivariable function 
\begin{equation} \label{eqn:Theta definition}
    \Theta(z) = \left( \theta_0(z) - 2\pi, \theta_1(z) - 2\pi, \ldots, \theta_9(z) - 2\pi \right) .
\end{equation}
We obtain the coordinates of $\hat S$ by truncating the $z$-coordinates at $10^{-32}$. 

\begin{figure}[ht]
    \centering
    
    \def\svgwidth{0.5\columnwidth}
    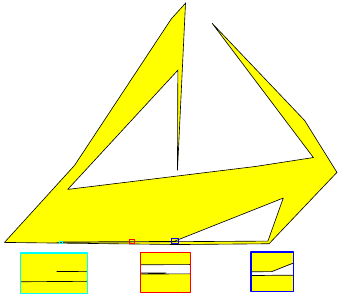

    \caption{The intersection of $\hat S$ with the plane $x = 0$, which is the boundary of a region that is not simply-connected. Again, the boxes contain the magnified view of their respective regions in the original slice.}
    \label{fig:slice_yz}
\end{figure}

\subsection{A 12-Vertex Model with 2-Fold Symmetry}

The candidate model $\hat S$ is rather unsatisfactory in the sense that it doesn't admit any symmetry. On the other hand, the 8-vertex paper torus model from \cite{schwartz2025vertexminimalpapertori} has a 2-fold symmetry that corresponds to a hyperelliptic involution, where the four fixed points are midpoints of 4 edges. Since every flat torus admits a hyperelliptic involution, Schwartz's construction effectively reduced the dimension of the search space to find a family of paper torus in the hope of realizing every flat structure. 
Very recently, Doyle and Schwartz \cite{doyle2025collapsibilitynearuniversalityvertex} found a family of 8-vertex flat polyhedral tori that realize any flat torus without reflection
symmetry, which constitutes a dense subset in the moduli space of flat tori. 

\begin{figure}
    \centering
    \includegraphics[width=0.3\columnwidth]{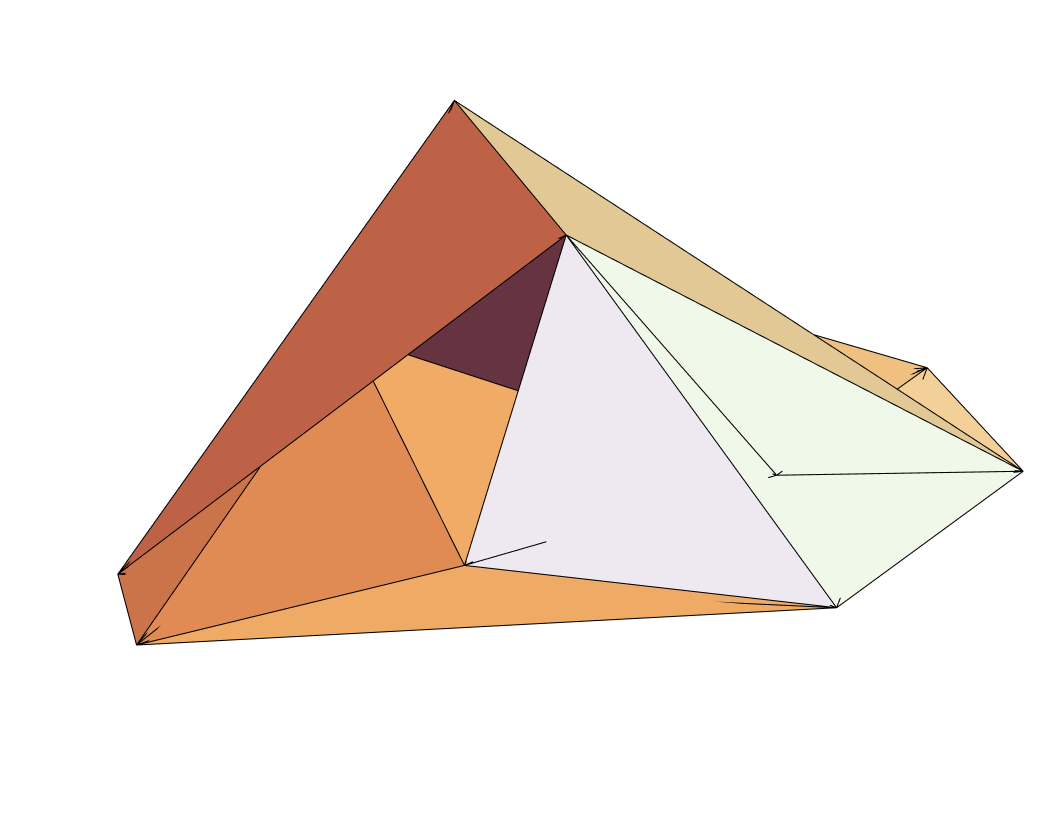}
    \includegraphics[width=0.3\columnwidth]{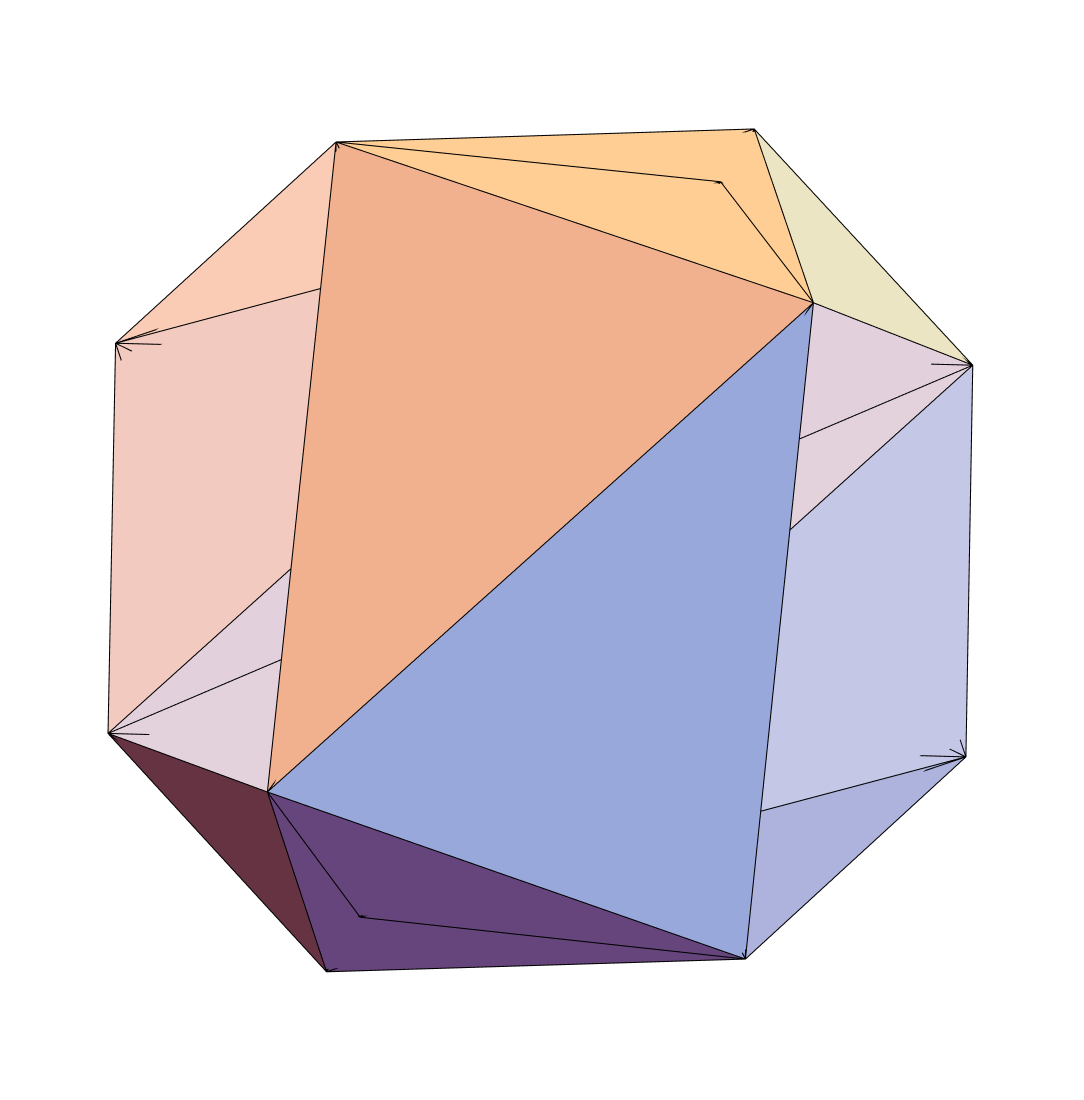}
    \includegraphics[width=0.3\columnwidth]{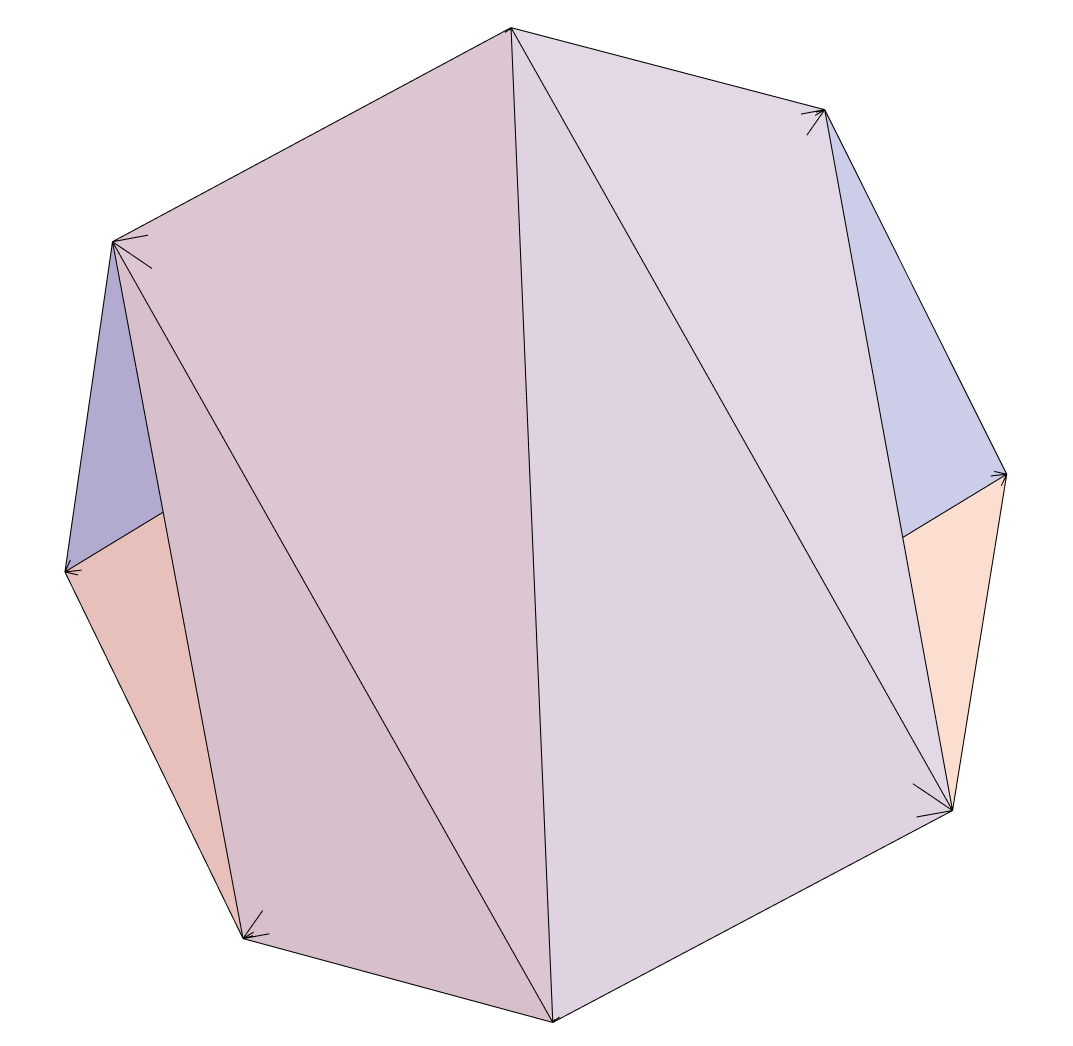}
    \caption{The embedding $\hat S_{12}$ with 2-fold rotational symmetry about the $z$-axis. Left: front view from the positive $x$-direction. Middle: bird's-eye view from the positive $z$-direction. Right: view from the negative $z$-direction.}
    \label{fig:12-vertex view}
\end{figure}

Motivated by this finding, I searched for a hyperbolic origami 2-torus with 2-fold symmetry that corresponds to a hyperelliptic involution. One first observes that if the order-2 symmetry acts freely on the set of vertices and fixes 6 points, then it must fix midpoints of 6 edges, which requires at least 12 vertices. I started by finding piecewise-linear embeddings of a 12-vertex triangulation of the 2-torus along with a candidate model with 2-fold symmetry into $\RR^3$. A natural starting point is to consider the 7-regular triangulations, which has the largest symmetry group. I first checked whether the 7-regular triangulations of the 2-torus from \cite{DattaUpadhyay+2006+1011+1025} admit piecewise-linear embeddings into $\RR^3$ with 2-fold symmetry. Then, I applied the same procedure as before to numerically approximate a hyperbolic origami 2-torus with 12 vertices and 2-fold symmetry. This led me to the following 12-vertex triangulation $\cT$ with vertices indexed by $\{0, \ldots, 11\}$. The indices of the 28 oriented triangles in $\cT$ are listed below:
\begin{equation} \label{eqn:12-vertex triangulation triangles}
    \begin{matrix}
        (0,1,5) & (0,4,8) & (0,5,11) & (0,6,1) & (0,7,6) & (0,8,7) & (0,11,4) \\
        (1,2,7) & (1,6,2) & (1,7,8)  & (1,8,9) & (1,9,5) & (2,3,7)  & (2,6,10) \\
        (2,8,3) & (2,9,8) & (2,10,9) & (3,4,9) & (3,8,4) & (3,9,10) & (3,10,11) \\
        (3,11,7) & (4,5,9) & (4,10,5) & (4,11,10) & (5,6,11) & (5,10,6) & (6,7,11)
    \end{matrix}
\end{equation}

Notice that $\cT$ admits a $\ZZ/12\ZZ$ symmetry generated by the map $\sigma: i \mapsto (i + 1) \mod 12$. The hyperelliptic involution is given by $\sigma^6$. Using this triangulation, I found a polyhedral embedding $\hat S_{12}: (\Sigma_2, \cT) \hookrightarrow \HH^3$ that satisfies $R \circ \hat S_{12} = \hat S_{12}$, where $R \in \Isom^+(\HH^3)$ is a rotation by $\pi$ about the $z$-axis. I checked that the cone angles of $\hat S_{12}$ are bounded away from $2\pi$ by $10^{-10}$. I provide the coordinates of $\hat S_{12}$ in $\S$\ref{subsec:12-vertex coords}. 

Figure \ref{fig:12-vertex view} contains the visualization of the embedding $\hat S_{12}$ in the Klein model. It is very curious that when we use the Klein model to visualize $\hat S_{12}$, its appearance resembles a ``double-deck pup tent'' (the \emph{pup tent} refers to the paper tori construction by Schwartz in \cite{schwartz2025vertexminimalpapertori}). This hints at the possibility to generalize this construction.

\begin{figure}
    \centering
    
    \def\svgwidth{0.8\columnwidth}
    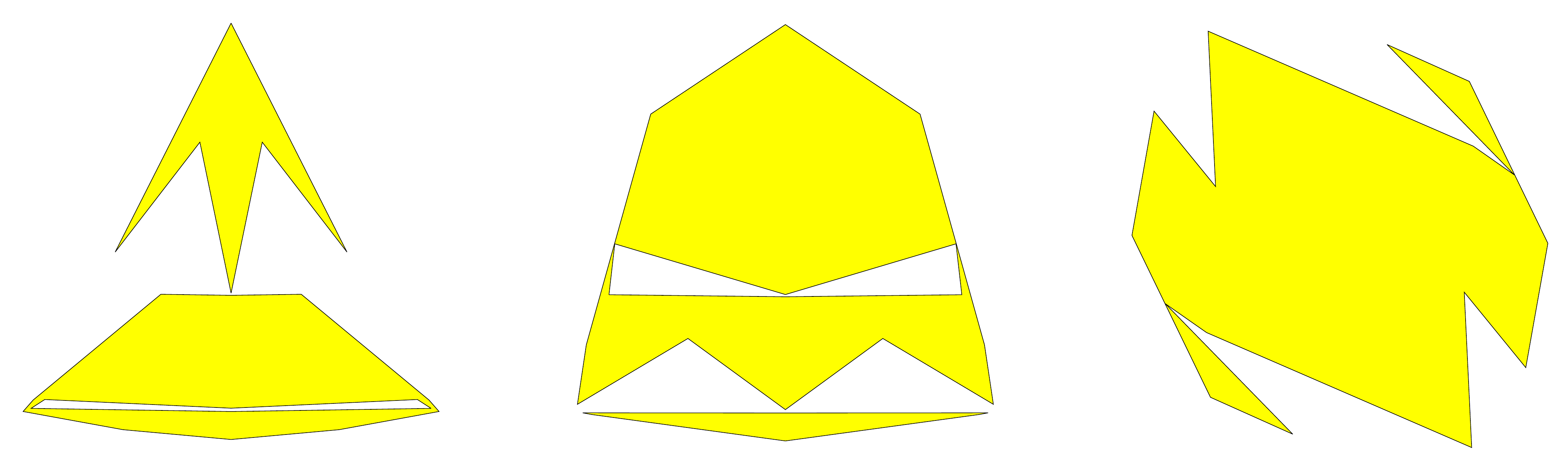

    \caption{The intersection of the 12-vertex model with the plane $x = 0$ (left), $y = 0$ (middle), and $z = -0.05$ (right). One can clearly see the six midpoints of edges $(i, i+6)$ that are fixed by the rotational symmetry from the $x = 0$ and $y = 0$ slices.}
    \label{fig:12-vertex slices}
\end{figure}

The midpoints of the six edges $(i, i+6)$ all lie on the $z$-axis, which can be seen by slicing the embedding with planes that contains the $z$-axis. In Figure \ref{fig:12-vertex slices}, we show the intersection of $\hat S_{12}$ with the planes $x = 0$, $y = 0$, and $z = -0.05$. One can clearly see the six fixed points of the hyperelliptic involution from the $x = 0$ and $y = 0$ slices. Finally, in Figure \ref{fig:12-vertex universal cover}, we show the induced triangle tiling of the universal cover $\HH^2$ from $\hat S_{12}$ in the Poincar\'e disk model. In this visualization, the pairs of triangles $(i, j, k)$ and $(i+6, j+6, k+6)$ are colored the same way to reflect the 2-fold symmetry.

\begin{figure}
    \centering
    
    \def\svgwidth{0.4\columnwidth}
    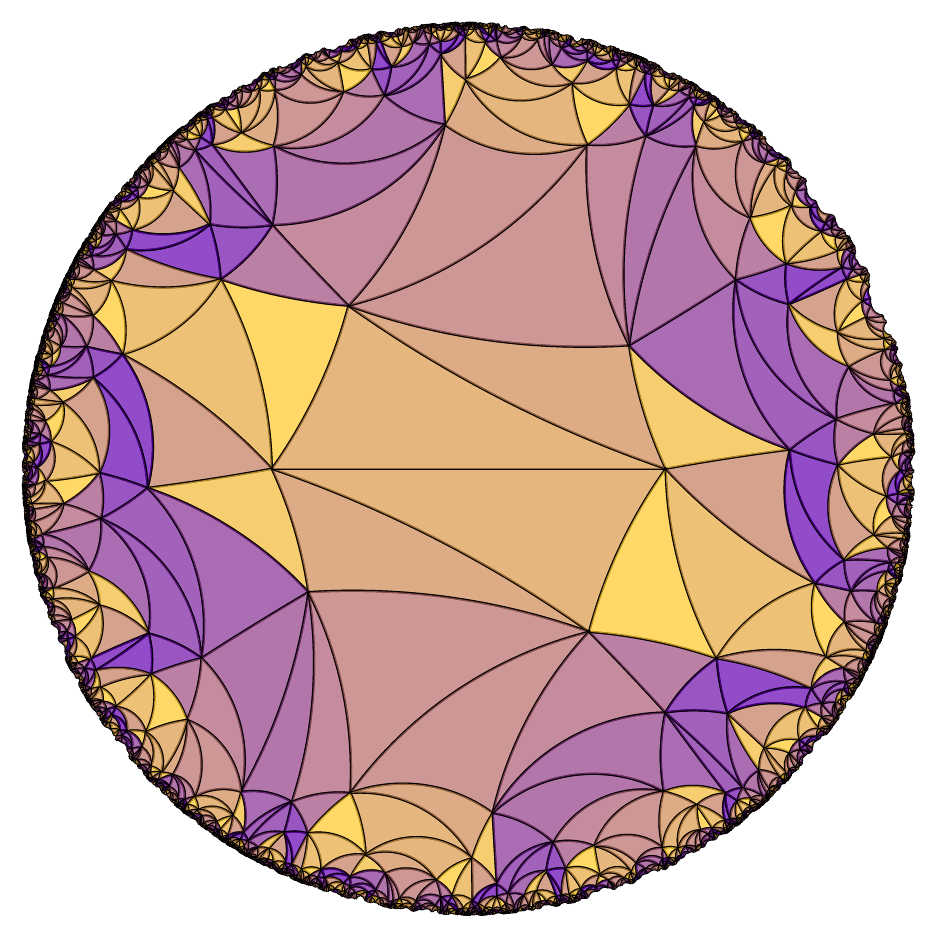

    \caption{The isometrically induced 7-regular triangle tiling in the universal cover $\HH^2$, visualized in the Poincar\'e disk model. The vertices $z_0$ and $z_6$ are fixed on the real axis, with $\Re(z_0) < 0$ and $z_6 = -z_0$.}
    \label{fig:12-vertex universal cover}
\end{figure}

\subsection{Paper Organization}

This paper is organized as follows:
In $\S$\ref{sec:background}, we remind the readers of the properties of the Klein model and explain how we incorporate floating-point arithmetic into the proofs. 
In $\S$\ref{sec:almost flat}, we show that the cone angles of $\hat S$ differ from $2\pi$ by at most $10^{-28}$, and that $\hat S$ stays embedded after perturbing its $z$-coordinates by at most $10^{-7}$. 
In $\S$\ref{sec:flat}, we analyze the Jacobian matrix of $\Theta(z)$ to show that there exists a hyperbolic origami 2-torus $S$ near $\hat S$. 

An ongoing project is to generalize the 12-vertex construction to a family of hyperbolic origami 2-tori and check which hyperbolic structures can be realized.

\subsection{Accompanying Programs}

I have written programs in both Python and Mathematica to visualize the candidate models and slicing them with planes. Readers can access them via the following link: 

\begin{center}
    \href{https://github.com/zzou9/hyperbolic-origami}{\textbf{https://github.com/zzou9/hyperbolic-origami}}
\end{center}

\subsection{Acknowledgements}

This project is supported by an N.S.F.\ Research Grant DMS-2505281. I would like to thank my advisor, Richard Schwartz, for inspiring me to work on this problem. Much of the methodology used in this paper is inspired by \cite{schwartz2025vertexminimalpapertori}. I would like to thank Peter Doyle, Fabian Lander, Alba M\'alaga, Stepan Paul, Steve Trettel, and Bena Tshishiku for helpful discussions on this topic. I would like to thank Frank Lutz for his manifolds webpage \cite{Lutz2017_manifoldpage} and Stefan Hougardy, Frank H.\,Lutz, Mariano Zelke for keeping a database of piecewise-affine embedded 2-tori with vertex-minimal triangulations, which helped me find prototypes for the candidate 2-torus in this paper. 

\section{Background} \label{sec:background}

\subsection{The Beltrami--Klein Model}

Denote by $\HH^n$ the hyperbolic $n$-space -- the simply connected Riemannian $n$-manifold of constant sectional curvature -1. We mainly work with $\HH^3$ in this paper. Here we provide the explicit construction known as the \emph{Beltrami--Klein model} of $\HH^3$. Consider the open unit Euclidean ball $B_1(0) \subset \RR^3$. For each $X \in B_1(0)$, we write $a_X = 1 - \|X\|^2$. The Riemannian manifold $\HH^3$ is defined as the manifold $B_1(0)$ along with the Riemannian metric $g$ given by 
\begin{equation} \label{eqn:klein 3 dot product}
    g|_{X} = \left( \frac{(dx^1)^2 + (dx^2)^2 + (dx^3)^2}{a_X} + \frac{(x^1 dx^1 + x^2 dx^2 + x^3 dx^3)^2}{a_X^2} \right)
\end{equation}
Let $V, W$ be two arbitrary vectors in $\RR^3$, which we can identify with $T_X \HH^3$. Then, $g|_X$ can be expressed as 
\begin{equation} \label{eqn:klein 3 dot product vectors}
    g|_X (V, W)
    = \frac{\langle V, W \rangle}{a_X} + \frac{\langle X, V \rangle \, \langle X, W \rangle}{a_X^2} 
    = \frac{a_X \langle V, W \rangle + \langle X, V \rangle \langle X, W \rangle}{a_X^2}.
\end{equation}
where $\langle \cdot , \cdot \rangle$ denotes the Euclidean dot product. We will write $\langle V, W \rangle_X = g|_X (V, W)$ and $\|V\|_X = \sqrt{g|_X(V, V)}$ for the norm of $V \in T_X\HH^n$. 

Since geodesics in $\HH^3$ are straight lines, a totally geodesic triangle in $\HH^3$ with vertices $(X, Y, Z)$ is the same as the Euclidean triangle in the unit ball. We may compute the hyperbolic angle $\theta$ at $X$ as follows: Let $V = Y - X$ and $W = Z- X$. We may identify $V$ and $W$ with velocity vectors in $T_X\HH^3$ of the geodesic from $X$ to $Y$ and $X$ to $Z$ respectively. Consider the following function: 
\begin{equation} \label{eqn:hyperbolic Klein dot product}
    A(X, Y, Z)
    = \frac{\langle V, W \rangle_X^2}{\langle V, V \rangle_X \cdot \langle W, W \rangle_X} 
    = \frac{(a_X \langle V, W \rangle + \langle X, V \rangle \langle X, W \rangle)^2}{(a_X \langle V, V \rangle + \langle X, V \rangle^2) (a_X \langle W, W \rangle + \langle X, W \rangle^2)}.
\end{equation}
Denote by $\sigma = \sgn(\langle V, W \rangle_X)$ the sign of $\langle V, W \rangle_X$. 
It follows that 
\begin{equation} \label{eqn:hyperbolic Klein angle}
    \theta 
    = \arccos \left( \frac{\langle V, W \rangle_X}{\|V\|_X \, \|W\|_X} \right)
    = \arccos\left( \sigma \sqrt{A(X, Y, Z)} \right) ; \hspace{10pt}
    \cos\theta = \sigma \sqrt{A(X, Y, Z)}.
\end{equation}

\subsection{Remark on Floating-Point Arithmetic}

The proofs in this paper make use of computer floating-point arithmetic. To perform floating-point arithmetic up to high precision ($10^{-400}$ for our purposes), we use the Python package mpmath, which stores floating-point numbers of any designated precision and has built-in methods of basic arithmetic operations. When we claim that $\frac{a}{b} < \frac{c}{d}$ for floating-point numbers $a, b, c, d$, we implicitly carry out the computation of $ad - bc$ in mpmath and check that the result is negative. We sometimes find equivalent fractions $\frac{a'}{b'} = \frac{a}{b}$ and $\frac{c'}{d'} = \frac{c}{d}$ with $a', b', c', d' \in \ZZ$ to simplify the computation into integer arithmetic. 

For finding the bounds on the value of $e^a$ where $a$ is a floating-point number, see $\S$\ref{subsec:transcendental functions}.

\section{The Candidate 2-Torus} \label{sec:almost flat}

In this section, we show that the cone angles of our candidate 2-torus $\hat S$ are very close to $2\pi$. We then show that if we apply perturbation to the $z$-coordinates of $\hat S$, the resulting 2-torus stays embedded as long as the perturbation is small enough. 

Let $\cT$ be a triangulation of the 2-torus $\Sigma_2$. A \emph{hyperbolic polyhedral 2-torus} is an embedding $S: (\Sigma_2, \cT) \hookrightarrow \HH^3$ such that the image of the triangles in $\cT$ under $S$ are geodesic triangles. A hyperbolic origami 2-torus is then a hyperbolic polyhedral 2-torus where the cone angle of each vertex is $2\pi$. 

\subsection{Robust Flatness}
\label{subsec:flatness}

We start by defining flatness of a hyperbolic polyhedral 2-torus. 

\begin{definition}
    We say that a hyperbolic polyhedral 2-torus $S$ is \emph{$\varepsilon$-flat} if $\max_i |\theta_i - 2\pi| < \varepsilon$, where $i \in \{0, \ldots, 9\}$ and $\theta_i$ is the cone angle of vertex $i$. 
\end{definition}

\begin{proposition} \label{prop:flatness of S hat}
    $\hat S$ is $10^{-28}$-flat. 
\end{proposition}

\begin{proof}
    Let $d_i$ be the degree of vertex $i$, and let $n_0, \ldots, n_{d_i-1}$ be the indices of the neighbors of $i$, ordered so that $(i, n_j, n_{j+1}) \in F$ for $j = 0, \ldots, d_i-1$ (here the subscripts $j$ and $j+1$ are taken modulo $d_i$). 
    We first find integer vectors $Y_{i, n_j} \in \RR^2$ for all $i = 0, \ldots, 9$ and $j = 0, \ldots, d_i - 1$ so that the difference $\delta_{i,j} = |\alpha_{i,j} - \beta_{i,j}|$ is minimized, where 
    \begin{equation*}
        \alpha_{i, j} = A(\hat X_i, \hat X_{n_j}, \hat X_{n_{j+1}}) ;
        \hspace{10pt}
        \beta_{i, j} = \frac{\langle Y_{i, n_j}, Y_{i, n_{j+1}} \rangle^2}{\langle Y_{i, n_j}, Y_{i, n_j} \rangle \langle Y_{i, n_{j+1}}, Y_{i, n_{j+1}} \rangle}
    \end{equation*}
    Here, $A$ is the map from Equation \eqref{eqn:hyperbolic Klein dot product}. The coordinates of $Y_{i, n_j}$ are given in $\S$\ref{subsec:flatness coords}. We obtained these coordinates by lifting the link of each vertex $\hat X_i$ isometrically into the Poincar\'e model of the universal cover $\HH^2$, with $\hat X_i \mapsto 0$ and $\hat X_{n_0}$ mapping to the positive real axis. Then, we scale up each coordinate by $10^{32}$ to make them all integer-valued. 
    We computed using floating-point arithmetic that $\max_{i, j} \delta_{i,j} \leq 2.93 \times 10^{-32}$. We also computed that $\alpha_{i,j} \in [0.000052, 0.918]$. It follows that $\alpha_{i,j}, \beta_{i,j} \in [0.00005, 0.92]$ for all $i,j$. In particular, $\langle X_{n_j} - X_i, X_{n_{j+1}} - X_i \rangle_{X_i}$ and $\langle Y_{i, n_j}, Y_{i, n_{j+1}} \rangle$ always share the same sign. 

    Next, we explain how we obtain the cone angle bounds from $\alpha_{i,j}$ and $\beta_{i,j}$. Consider two functions $\phi_+$ and $\phi_-$ given by $\phi_+(x) = \arccos(\sqrt{x})$ and $\phi_-(x) = \arccos(-\sqrt{x})$. Denote by $\theta_{i,j}$ the hyperbolic angle at $\hat X_i$ of the triangle $(\hat X_i, X_{n_j}, \hat X_{n_{j+1}})$, and denote by $\tau_{i,j}$ the Euclidean angle between the two vectors $Y_{i, n_j}$ and $Y_{i, n_{j+1}}$. Since $\langle X_{n_j} - X_i, X_{n_{j+1}} - X_i \rangle_{X_i}$ and $\langle Y_{i, n_j}, Y_{i, n_{j+1}} \rangle$ always share the same sign, we know that 
    \begin{itemize}
        \item $\theta_{i,j} = \phi_+(\alpha_{i,j})$ and $\tau_{i,j} = \phi_+(\beta_{i,j})$ if $\langle X_{n_j} - X_i, X_{n_{j+1}} - X_i \rangle_{X_i} > 0$;
        \item $\theta_{i,j} = \phi_-(\alpha_{i,j})$ and $\tau_{i,j} = \phi_-(\beta_{i,j})$ if $\langle X_{n_j} - X_i, X_{n_{j+1}} - X_i \rangle_{X_i} < 0$.
    \end{itemize}
    Next, observe that 
    \begin{equation*}
        \frac{d\phi_+}{dx} = \frac{-1}{2\sqrt{x(1 - x)}}; \hspace{10pt}
        \frac{d\phi_-}{dx} = \frac{1}{2\sqrt{x(1 - x)}}
    \end{equation*}
    One can compute directly that $| \frac{d\phi_+}{dx} |, | \frac{d\phi_-}{dx} | \leq 71$ on $[0.00005, 0.92]$, so both $\phi_+$ and $\phi_-$ are $71$-Lipschitz on $[0.00005, 0.92]$. Finally, we have 
    \begin{equation*}
        \begin{aligned}
            | \theta_i - 2\pi | 
            &= \left| \sum_{j=0}^{d_i-1} \theta_{i,j} - \sum_{j=0}^{d_i-1} \tau_{i,j} \right|
            \leq \sum_{j=0}^{d_i-1} |\theta_{i,j} - \tau_{i,j}| 
            \leq \max_i d_i \cdot \max_{i,j} | \theta_{i,j} - \tau_{i,j} | \\
            &\leq \max_i d_i \cdot 71 \cdot \max_{i,j} \delta_{i,j}
            \leq 71 \times 9 \times 2.93 \times 10^{-32} < 10^{-28}.
        \end{aligned}
    \end{equation*}
    Taking $\varepsilon = 10^{-28}$ completes the proof. 
\end{proof}

\subsection{Robust Embeddedness}\label{subsec:embeddedness}

The next step is to show that if we slightly perturb the $z$-coordinates of $\hat S$, the resulting 2-torus is still embedded. The key observation (and the main reason we use the Beltrami--Klein model of $\HH^3$) is that the underlying set of the hyperbolic geodesic triangle with vertices $V_1, V_2, V_3 \in \HH^3$ is precisely the Euclidean triangle with the same vertices. Therefore, in this section we ``forget'' the hyperbolic structure and instead treat $\hat S$ as a map $\hat S: (\Sigma_2, \cT) \rightarrow B_1(0)$ whose image is the unit Euclidean ball. We begin with the following definition, which is not “natural” in the hyperbolic setting but is more appropriate for Euclidean geometry.

\begin{definition} \label{def:robust embeddedness}
    A map $S: (\Sigma_2, \cT) \rightarrow B_1(0)$ is \emph{$\lambda$-robustly embedded} if for all $S'$ obtained from moving each of the $z$-coordinates of the vertices of $S$ by at most $\lambda$ is also embedded.
\end{definition}

\begin{remark}
    Readers may have noticed that Definition \ref{def:robust embeddedness} is not invariant under hyperbolic isometries. That is, suppose $S: (\Sigma_2, \cT) \hookrightarrow \HH^3$ is $\lambda$-robustly embedded and $M \in \Isom(\HH^3)$, then there is no reason whatsoever to expect that $M \circ S$ must also be $\lambda$-robustly embedded. One should keep in mind that throughout this section, we will ``pretend to'' work with Euclidean geometry. 
\end{remark}

To show embeddedness, it suffices to consider pairs of triangles with no common vertices and pairs that share one vertex in common. Below is a lemma from \cite{schwartz2025vertexminimalpapertori}. We present a self-contained proof here. 

\begin{lemma}[{\cite[Lemma 3.2]{schwartz2025vertexminimalpapertori}}] \label{lem:embeddedness check for triangles}
    Given a map $S: (\Sigma_2, \cT) \rightarrow B_1(0)$ sending triangles in $\cT$ to Euclidean triangles in $B_1(0)$. If all pairs of triangles of $S$ sharing no vertices in common are disjoint and all pairs of triangles sharing one vertex only intersect at the shared vertex, then $S$ is embedded.
\end{lemma}

\begin{proof}
    Consider the pairs of triangles $(i, j, k)$ and $(j, i, l)$ that share two vertices. Denote by $e_{ab}$ the edge connecting $X_a$ to $X_b$, where $a,b \in \{i, j, k, l\}$. We first recall that every edge is contained in two triangles, so the edge $e_{ik}$ is contained in another triangle $(i, k, m)$, where $m \neq j$. If the interior of $e_{ik}$ intersects $(j, i, l)$, then $(j, i, l)$ and $(i, k, m)$ intersect at a point other than $X_i$. Notice that $l \neq m$, for otherwise the degree of vertex $i$ would equal 3. This implies $(j, i, l)$ and $(i, k, m)$ only share one vertex, but that contradicts our assumption in the lemma. The same argument shows the interior of $e_{jk}$ doesn't intersect $(j, i, l)$, and the interior of $e_{il}$ and $e_{jl}$ doesn't intersect $(i, j, k)$, so these two triangles only intersect at $e_{ij}$, which shows that $S$ is embedded. 
\end{proof}

We first deal with the pairs of disjoint triangles. 
We say that two disjoint triangles $T_1$ and $T_2$ in $\RR^3$ are \emph{$\delta$-separated} if they remain disjoint after simultaneously perturbing the $z$-coordinates of the vertices by at most $\delta$. Here is a sufficient hyperplane separation test to show that $T_1$ and $T_2$ are $\delta$-separated. 

\begin{lemma} \label{lem:separation test for disjoint triangles}
    Two disjoint triangles $T_1$ and $T_2$ are $\delta$-separated if there exists a vector $\hat N$ with $\| \hat N \|_{\infty} < C$ for some $C \in \RR$, such that 
    \begin{equation} \label{eqn:separation test for disjoint triangles 1}
        \min_{X \in T_1} \langle X, \hat N \rangle 
        \, - \, \max_{Y \in T_2} \langle Y, \hat N \rangle 
        \, > \, 2\delta C .
    \end{equation}
    In particular, if $T_1$ has vertices $V_1, V_2, V_3$ and $T_2$ has vertices $W_1, W_2, W_3$, then it suffices to check that 
    \begin{equation} \label{eqn:separation test for disjoint triangles 2}
        \min_{X \in \{V_1, V_2, V_3\}} \langle X, \hat N \rangle 
        \, - \, \max_{Y \in \{W_1, W_2, W_3\}} \langle Y, \hat N \rangle 
        \, > \, 2\delta C .
    \end{equation}
\end{lemma}

\begin{proof}
    We know that every $X \in T$ can be expressed as $X = t_1 V_1 + t_2 V_2 + t_3 V_3$ where $t_i \geq 0$ and $t_1 + t_2 + t_3 = 1$. Linearity of the dot product then implies the minimum of $\langle X, \hat N \rangle$ must be achieved at one of the vertices of $T$, so Equation \eqref{eqn:separation test for disjoint triangles 2} implies \eqref{eqn:separation test for disjoint triangles 1}. Now, if $T'_1$ and $T'_2$ are obtained from purturbing the $z$-coordinates of the vertices of $T_1$ and $T_2$ by at most $\delta$, then the dot product $\langle V_i, \hat N \rangle$ is perturbed by at most $\delta C$. It follows that 
    \begin{equation*}
        \min_{X' \in T_1'} \langle X', \hat N \rangle - \max_{Y' \in T_2'} \langle Y', \hat N \rangle 
        \geq 
        \min_{X \in T_1} \langle X, \hat N \rangle - \max_{Y \in T_2} \langle Y, \hat N \rangle - 2 \delta C
        > 0 .
    \end{equation*}
    This implies $T_1' \cap T_2' = \emptyset$, which completes the proof. 
\end{proof}

Next, we deal with triangles that share one of their vertices. Suppose $(U, V_1, V_2, W_1, W_2)$ are five points in $\RR^3$, and $T_1$, $T_2$ are two triangles with vertices $(U, V_1, V_2)$ and $(U, W_1, W_2)$ respectively. We say that $T_1$ and $T_2$ are \emph{$\delta$-separated} if for all $T'_1$ and $T'_2$ with vertices $(U', V'_1, V'_2)$ and $(U', W'_1, W'_2)$ such that $(U', V'_1, V'_2, W'_1, W'_2)$ are obtained from purturbing the $z$-coordinates of $(U, V_1, V_2, W_1, W_2)$ by at most $\delta$, $T'_1 \cap T'_2 = \{U'\}$. 

\begin{lemma} \label{lem:separation test for triangles sharing one vertex}
    $T_1$ and $T_2$ are $\delta$-separated if there exists a vector $\hat N$ with $\| \hat N \|_{\infty} < C$ for some $C \in \RR$, such that 
    \begin{equation} \label{eqn:separation test for triangles sharing one vertex}
        \min_{X \in \{V_1, V_2\}} \langle X, \hat N \rangle - \langle U, \hat N \rangle > 2\delta C 
        \hspace{5pt} \text{and} \hspace{5pt} 
        \langle U, \hat N \rangle - \min_{Y \in \{W_1, W_2\}} \langle Y, \hat N \rangle> 2\delta C .
    \end{equation}
\end{lemma}

\begin{proof}
    Suppose $T'_1$ and $T'_2$ are obtained from purturbing $(U, V_1, V_2, W_1, W_2)$ by at most $\delta$. The dot products with $\hat N$ is perturbed by at most $\delta C$. It follows that 
    \begin{equation} \label{eqn:proof of separation test for triangles sharing one vertex}
        \min_{X' \in \{V'_1, V'_2\}} \langle X', \hat N \rangle - \langle U', \hat N \rangle > 0
        \hspace{5pt} \text{and} \hspace{5pt} 
        \langle U', \hat N \rangle - \min_{Y' \in \{W'_1, W'_2\}} \langle Y', \hat N \rangle> 0.
    \end{equation}
    To see this implies $T'_1 \cap T'_2 = \{U'\}$, suppose there exists $X' \in T'_1 \cap T'_2$. We may express $X' = t_0 U + t_1 V'_1 + t_2 V'_2 = s_0 U + s_1 W'_1 + s_2 W'_2$ where $t_0 + t_1 + t_2 = s_0 + s_1 + s_2 = 1$ and all coefficients $t_i, s_j \geq 0$. It follows that 
    \begin{equation*}
        \langle X', \hat N \rangle - \langle U', \hat N \rangle 
        = t_1 \langle V'_1 - U', \hat N \rangle + t_2 \langle V'_2 - U', \hat N \rangle
    \end{equation*}
    and 
    \begin{equation*}
        \langle U', \hat N \rangle - \langle X', \hat N \rangle
        = s_1 \langle U' - W'_1, \hat N \rangle + s_2 \langle U' - W'_2, \hat N \rangle.
    \end{equation*}
    This implies 
    \begin{equation*}
        t_1 \langle V'_1 - U', \hat N \rangle 
        + t_2 \langle V'_2 - U', \hat N \rangle
        + s_1 \langle U' - W'_1, \hat N \rangle 
        + s_2 \langle U' - W'_2, \hat N \rangle
        = 0,
    \end{equation*}
    so Equation \eqref{eqn:proof of separation test for triangles sharing one vertex} implies $t_1 = t_2  = s_1 = s_2 = 0$, which gives us $X' = U'$ as desired. 
\end{proof}

We are now ready to prove that $\hat S$ is robustly embedded.

\begin{proposition} \label{prop:embeddedness of S hat}
    $\hat S$ is $10^{-7}$-robustly embedded.
\end{proposition}

\begin{proof}
    We dilate the surface $\hat S$ by a factor of $10^{32}$ so that every vertex has integer coordinates. Call the dilated surface $\hat S'$. 
    Lemma \ref{lem:separation test for disjoint triangles} and \ref{lem:separation test for triangles sharing one vertex} tells us that it suffices to find normal vectors $\hat N$ for each pair of triangles in $\hat S'$. 

    Among the 276 pairs of triangles from $F$, there are 82 pairs that share no vertices in common, and 158 of which share one vertex in common. We provide an algorithm to generate $\hat N$ for each of these pairs where the conditions in Lemma \ref{lem:separation test for disjoint triangles} and \ref{lem:separation test for triangles sharing one vertex} hold. Consider a map $\rho: \ZZ_{>0} \rightarrow \ZZ^3$ as follows: 
    \begin{equation*}
        \rho(n) = \left( 
            \lfloor 10^5 \cdot L(n \sqrt{2}) \rfloor ,  
            \lfloor 10^5 \cdot L(n \sqrt{3}) \rfloor ,
            \lfloor 10^5 \cdot L(n \sqrt{5}) \rfloor
        \right) ; \hspace{10pt}
        L(x) = 2 (x - \lfloor x \rfloor) - 1
    \end{equation*}
    It is clear that $\| \rho(n) \|_{\infty} < 10^5$ for all $n$. We found indices $n < 2000$ such that the conditions in Lemma \ref{lem:separation test for disjoint triangles} hold for $\delta = 10^{25}$ by taking $\hat N = \pm \rho(n)$ for all but the following two pairs of disjoint triangles, where we manually found the candidates for $\hat N$:
    \begin{equation*}
        \begin{aligned}
            \{ (2, 7, 4), (1, 8, 3) \} , \hspace{10pt} & \hat N = (-35, -74, 12106) ; \\
            \{ (0, 2, 1), (3, 4, 7) \} , \hspace{10pt} & \hat N = (-60, -96, 22534) .
        \end{aligned}
    \end{equation*}
    We also computed $\rho(n)$ for the triangle pairs that share one vertex. We found indices $n < 10^5$ such that the conditions in Lemma \ref{lem:separation test for triangles sharing one vertex} hold for $\delta = 10^{25}$ by taking $\hat N = \pm \rho(n)$ for all but the following nine pairs, where we manually found the candidates for $\hat N$:
    \begin{equation*}
        \begin{aligned}
            \{ (0, 2, 1), (1, 3, 7) \} , \hspace{10pt} & \hat N = (-40, -67, 14035) ; \\
            \{ (0, 2, 1), (2, 7, 4) \} , \hspace{10pt} & \hat N = (73, 97, -22643) ; \\
            \{ (1, 2, 4), (1, 3, 7) \} , \hspace{10pt} & \hat N = (-54, -85, 14773) ; \\
            \{ (1, 2, 4), (1, 8, 3) \} , \hspace{10pt} & \hat N = (-120, -151, 27015) ; \\
            \{ (1, 2, 4), (3, 4, 7) \} , \hspace{10pt} & \hat N = (45, 69, -14773) ; \\
            \{ (1, 3, 7), (2, 3, 8) \} , \hspace{10pt} & \hat N = (-184, -155, -15412) ; \\
            \{ (1, 3, 7), (2, 7, 4) \} , \hspace{10pt} & \hat N = (-36, -73, 12086) ; \\
            \{ (1, 8, 3), (3, 4, 7) \} , \hspace{10pt} & \hat N = (-35, -74, 12107) ; \\
            \{ (2, 3, 8), (3, 4, 7) \} , \hspace{10pt} & \hat N = (-417, 566, 51293) .
        \end{aligned}
    \end{equation*}
    Lemma \ref{lem:embeddedness check for triangles} then implies the $10^{32}$-dilated surface $\hat S'$ is $10^{25}$-robustly embedded, so the original surface $\hat S$ is $10^7$-robustly embedded, as desired.
\end{proof}

\section{Finding Origami 2-Tori} \label{sec:flat}

Our method of finding the flat 2-tori is analogous to Schwartz's method in \cite{schwartz2025vertexminimalpapertori}. The idea is to use the inverse function theorem to show that there exists a neighborhood around $\hat z$ whose image under the map $\Theta$ in Equation \eqref{eqn:Theta definition} hits the origin. 

\begin{definition}
    Let $f: \RR^n \rightarrow \RR^n$ be a smooth map. We say that $f$ is \emph{$\lambda$-expansive} on a subset $A \subset \RR^n$ if for all unit vectors $V \in \RR^{n}$ and $p, q \in A$, we have 
    \begin{itemize}
        \item $\| df|_p(V) \| > 2\lambda$;
        \item The angle between $df|_p(V)$ and $df|_q(V)$ is less than $\frac{\pi}{3}$.
    \end{itemize}
\end{definition}

The following lemma is proved by Schwartz in \cite{schwartz2025vertexminimalpapertori}. Here we provide an alternative self-contained proof. 

\begin{lemma}[{\cite[Lemma 3.3]{schwartz2025vertexminimalpapertori}}] \label{lem:inverse function theorem expansion}
    Let $B_r(p) \subset \RR^n$ denote the open ball of radius $r$ centered at $p$. 
    If $f: \RR^{n} \rightarrow \RR^{n}$ is $\lambda$-expansive on $B_r(p)$ for some $\lambda > 0$, then $B_{\lambda r}(f(p)) \subset f(B_r(p))$.
\end{lemma}

\begin{proof}
    After normalizing by translation and scaling we may assume $f(p) = p = 0$, $\lambda = r = 1$. Since $\|df|_q(V)\| > 0$ for all $q \in B_1(0)$ and unit vectors $V$, $df|_q$ is nondegenerate, so by the inverse function theorem, $f$ is an open map on $B_1(0)$. It follows that $\RR^n \setminus f(B_1(0))$ is closed. Consider the set $A = \overline{B_s(0)} \setminus f(B_1(0))$, where $s = 2 \sup_{x \in B_1(0)} \|f(x)\|$. We know that $A$ is compact, and $A \cup f(B_1(0)) = \overline{B_s(0)}$. The statement is true if $\|x\| \geq 1$ for all $x \in A$. 

    Since $A$ is compact, $\inf_{x \in A} \|x\|$ is achieved by some $\hat x \in A$. Assume for contradiction that $\| \hat x \| = \alpha$ for some $\alpha < 1$. Then, for all $n \in \NN$, there exists $x_n \in B_1(0)$ such that $f(x_n) \in B_{1/n}(\hat x)$. It follows that $\lim_{n \rightarrow \infty} f(x_n) = \hat x$. On the other hand, let $\beta[0,1] \rightarrow \RR^n$ be the unit-speed parametrization of the line segment from $0$ to $x_n$, and let $\gamma = f \circ \beta$. Since $\|df|_{\gamma(t)}(\beta'(t))\| > 2$, we know that $\gamma$ has length at least $2 \|x_n\|$. Also, since the angle between $df|_{\gamma(t)}(\beta'(t))$ and $df|_0(\beta'(t))$ is $< \frac{\pi}{3}$, the projection of $\gamma(t)$ onto the line connecting $0$ to $f(x_n)$ is a curve that moves at a speed greater than 1. It follows that $\|f(x_n)\| > \|x_n\|$, so $\|x_n\| < \frac{1+\alpha}{2}$ for $n$ sufficiently large. This implies the sequence $\{x_n\}$ lives in a compact subset of $B_1(0)$, so it converges to some $y \in B_1(0)$ on a subsequence, but then $f(y) = \hat x$, contradicting $\hat x \in A$. 
\end{proof}

Denote by $\cB$ the set of embeddings $S$ whose $z$-coordinates lie in $B_{10^{-18}}(\hat z)$, where $\hat z$ is the vector of $z$-coordinates of the almost-flat 2-torus $\hat S$ from $\S$\ref{sec:almost flat}.

\begin{proposition} \label{prop:expansion}
    The map $\Theta: \RR^{10} \rightarrow \RR^{10}$ is $\frac{1}{2}$-expansive on $B_{10^{-18}}(\hat z)$. 
\end{proposition}

\begin{proof}[Proof of Theorem \ref{thm:flat 2-torus for all vertices} assuming Proposition \ref{prop:expansion}]
    From Proposition \ref{prop:flatness of S hat} we know that $\hat S$ is $10^{-28}$-flat, so $\| \Theta(\hat z) \| < 10^{-27}$, or $0 \in B_{10^{-27}}(\Theta(\hat z))$. Then, Proposition \ref{prop:expansion} implies the existence of a hyperbolic origami 2-torus $S$ with coordinates $z \in B_{5 \times 10^{-19}}(\hat z)$ such that $\Theta(z) = 0$. In particular, Proposition \ref{prop:embeddedness of S hat} implies $S$ is embedded. This yields the hyperbolic origami 2-torus as desired. 
\end{proof}

The rest of this section will be devoted to proving Proposition \ref{prop:expansion}. 

\subsection{The Jacobian} \label{subsec:1st order partial}

Our first goal is to analytically compute the Jacobian matrix of the map $\Theta$. Recall from Equation \eqref{eqn:Theta definition} that $\Theta$ is given by 
\begin{equation*}
    \Theta_i(z) = \sum_{(i, j, k) \in F} \theta_{ijk}(z) - 2\pi ,
\end{equation*}
where $\theta_{ijk}(z)$ denotes the angle of the triangle $(i,j,k)$ at vertex $i$. 

We start by computing the first-order partials of $\theta_{ijk}$. 
Given $X_i = (x_i, y_i, z_i)$, denote by $\partial_i$ the operator $\pdv{}{z_i}$. Consider a triangle $(i, j, k) \in F$. We again write $V = X_j - X_i$, $W = X_k - X_i$, and $a_{X_i} = 1 - \|X_i\|^2$. We would like to compute the action of $\partial_i, \partial_j, \partial_k$ on $\langle V, W \rangle_{X_i}$, $\|V\|_{X_i}$, and $\|W\|_{X_i}$. By symmetry of the Riemannian metric, it suffices to compute $\partial_i$ and $\partial_j$. We first observe that 
\begin{equation*}
    \begin{matrix}
        \partial_i \langle X_i, X_i \rangle = 2z_i & 
        \partial_i \langle V, V \rangle = 2z_i - 2z_j & 
        \partial_i \langle W, W \rangle = 2z_i - 2z_k \\
        \partial_i \langle X_i, V \rangle = z_j - 2z_i & 
        \partial_i \langle X_i, W \rangle = z_k - 2z_i &
        \partial_i \langle V, W \rangle = 2z_i -z_j - z_k \\
        \partial_j \langle V, X_i \rangle = z_i &
        \partial_j \langle V, W \rangle = z_k - z_i &
        \partial_j \langle V, V \rangle = 2z_j - 2z_i
    \end{matrix}
\end{equation*}
It follows that
\begin{equation} \label{eqn:partial i dot VW}
    \begin{aligned}
        &\partial_i \langle V, W \rangle_{X_i}
        = \partial_i \left( \frac{\langle V, W \rangle}{a_{X_i}} + \frac{\langle X_i, V \rangle \langle X_i, W \rangle}{a_{X_i}^2} \right) \\
        = & \frac{2z_i - z_j - z_k}{a_{X_i}} + 
        \frac{2z_i \langle V, W \rangle + (z_j - 2z_i) \langle X_i, W \rangle + (z_k - 2z_i) \langle X_i, V \rangle}{a_{X_i}^2} + \frac{4z_i \langle X_i, V \rangle \langle X_i, W \rangle}{a_{X_i}^3}
    \end{aligned}
\end{equation}
and
\begin{equation} \label{eqn:partial j dot VW}
    \partial_j \langle V, W \rangle_{X_i} 
    = \partial_j \left( \frac{\langle V, W \rangle}{a_{X_i}} + \frac{\langle X_i, V \rangle \langle X_i, W \rangle}{a_{X_i}^2} \right) 
    = \frac{z_k - z_i}{a_{X_i}} + \frac{z_i \langle X_i, W \rangle}{a_{X_i}^2}.
\end{equation}

We also have 
\begin{equation} \label{eqn:partial i norm V}
    \begin{aligned}
        \partial_i \|V\|_{X_i}
        &= \partial_i \sqrt{\langle V, V \rangle_{X_i}} 
        = \frac{1}{2 \|V\|_{X_i}} \partial_i \langle V, V \rangle_{X_i} \\
        &= \frac{1}{\|V\|_{X_i}}
        \left( \frac{z_i - z_j}{a_{X_i}} + 
        \frac{z_i \langle V, V \rangle + (z_j - 2z_i) \langle X_i, V \rangle}{a_{X_i}^2} + \frac{2z_i \langle X_i, V \rangle^2}{a_{X_i}^3} \right) 
    \end{aligned}
\end{equation}
and
\begin{equation} \label{eqn:partial j norm V}
    \partial_j \|V\|_{X_i}
    = \frac{1}{2 \|V\|_{X_i}} \partial_j \langle V, V \rangle_{X_i} 
    = \frac{1}{\|V\|_{X_i}}
    \left( \frac{z_j - z_i}{a_{X_i}} + \frac{z_i \langle X_i, V \rangle}{a_{X_i}^2} \right) .
\end{equation}

Therefore, for all $l \in \{i,j,k\}$, we have
\begin{equation} \label{eqn:partial angle}
    \begin{aligned}
        \partial_l \theta_{ijk}
        &= \partial_l \arccos \left( \frac{\langle V, W \rangle_{X_i}}{\|V\|_{X_i} \|W\|_{X_i}} \right) 
        = \frac{- \|V\|_{X_i} \|W\|_{X_i}}{\sqrt{\langle V,V \rangle_{X_i} \langle W, W \rangle_{X_i} - \langle V, W \rangle_{X_i}^2}} \cdot \partial_l \left( \frac{\langle V, W \rangle_{X_i}}{\|V\|_{X_i} \|W\|_{X_i}} \right) \\
        &= \frac{\partial_l (\|V\|_{X_i} \|W\|_{X_i}) \langle V, W \rangle_{X_i}}{\|V\|_{X_i} \|W\|_{X_i} \cdot \sqrt{\langle V,V \rangle_{X_i} \langle W, W \rangle_{X_i} - \langle V, W \rangle_{X_i}^2}} - 
        \frac{\partial_l \langle V, W \rangle_{X_i}}{\sqrt{\langle V,V \rangle_{X_i} \langle W, W \rangle_{X_i} - \langle V, W \rangle_{X_i}^2}} .
    \end{aligned}
\end{equation}

Next, consider the following 10-by-10 matrix $M$: 
\begin{equation*}
        \begin{pmatrix}
        -7.526 & 0.793 & -0.228 & -0.264 & -0.782 & 1.372 & 1.173 & -0.098 & 0.192 & 2.158 \\
        0.684 & 4.991 & -0.344 & -0.104 & -1.468 & -0.260 & 0.208 & -0.661 & -1.528 & 0 \\
        -0.203 & -0.356 & 3.913 & -1.008 & -0.104 & 0 & -0.091 & -0.588 & -0.017 & 0.224 \\
        -0.255 & -0.116 & -1.093 & 4.665 & -0.847 & -0.035 & 0.117 & -0.133 & -0.037 & 0 \\
        -0.753 & -1.639 & -0.112 & -0.843 & 5.270 & -0.642 & 0.820 & 0.015 & 0 & 0.870 \\
        1.112 & -0.244 & 0 & -0.029 & -0.540 & -2.245 & 0.162 & 0 & -0.369 & 1.354 \\
        1.006 & 0.207 & -0.087 & 0.104 & 0.730 & 0.172 & -9.148 & 0 & 0 & 3.639 \\
        -0.086 & -0.674 & -0.580 & -0.121 & 0.013 & 0 & 0 & 5.063 & -1.681 & 0 \\
        0.128 & -1.185 & -0.013 & -0.026 & 0 & -0.305 & 0 & -1.277 & 3.455 & 0 \\
        1.751 & 0 & 0.204 & 0 & 0.733 & 1.355 & 3.444 & 0 & 0 & -11.599
    \end{pmatrix}
\end{equation*}

We state the following lemma, which we will prove in the next section.

\begin{lemma} \label{lem:bounding second order partial}
    For all $z \in B_{10^{-18}}(\hat z)$ and $i,j,k \in \{1, \ldots, 10\}$, we have 
    \begin{equation*}
        \left| \frac{\partial^2 \Theta_i}{\partial z_j \partial z_k} (z) \right| \leq 10^{14}.
    \end{equation*}
\end{lemma}

\begin{corollary}
    For all $z \in B_{10^{-18}}(\hat z)$, we have 
    \begin{equation*}
        d\Theta|_z = M + E,
    \end{equation*}
    where $M$ is the matrix above and $E$ is a matrix with $\|E\|_{\infty} < 0.002$. 
\end{corollary}

\begin{proof}
    Let $r = 10^{-18}$.
    One can use high precision floating-point arithmetic to compute $d \Theta|_{\hat z}$ and check that $d\Theta|_{\hat z} = M + E'$, where $\|E'\|_{\infty} < 0.001$. It then suffices to show that $d\Theta|_z = d\Theta|_{\hat z} + E''$, where $\| E'' \|_{\infty} < 0.001$. To see this, notice that $\| z - \hat z \| < r$. It follows that 
    \begin{equation*}
        \left| \frac{\partial \Theta_i}{\partial z_j}(z) - \frac{\partial \Theta_i}{\partial z_j}(\hat z) \right|
        \leq 10 \cdot r \cdot \sup_{k \in [10], w \in B} \left| \frac{\partial^2 \Theta_i}{\partial z_j \partial z_k} (w) \right|
        < 10^{-3}.
    \end{equation*}
    This finishes the proof. 
\end{proof}

\begin{proof}[Proof of Proposition \ref{prop:expansion}]
    We first claim that $\| M(V) \| > 1.5$ for all unit vectors $V$. To see this, one can numerically show that the smallest root of the characteristic polynomial of $M^t M$ is greater than $2.25$, where $M^t$ denotes the transpose of $M$. It then follows that the smallest singular value of $M$ is greater than $1.5$ (see Remark \ref{rmk:bounding singular values}). 
    
    Next, given a 10-by-10 matrix $E$ with $\|E\|_{\infty} < 0.002$, we must have 
    \begin{equation*}
        \|E\|_2 \leq \|E\|_F \leq 10^2 \|E\|_{\infty} < 0.2.
    \end{equation*}
    Here, $\|E\|_2 = \sup_{\|V\|=1} \|E(V)\|$ is the spectral norm of $E$ and $\|E\|_F$ is the Frobenius norm. 
    It follows that
    \begin{equation*}
        \| d\Theta|_z (V) \| 
        = \| M(V) + E(V) \| 
        \geq | \|M(V)\| - \|E(V)\| | 
        \geq 1.5 - 0.2
        > 1.
    \end{equation*}
    This shows that $\| d\Theta|_z(V) \| > 2 \lambda$. 

    To see that the angle $\theta$ between $d \Theta|_z(V)$ and $d \Theta|_w (V)$ is less than $\frac{\pi}{3}$ for all $z,w \in B_{10^{-18}}(\hat z)$, we first observe that $d \Theta|_w (V) = d \Theta|_z (V) + E'(V)$ where $\| E'(V) \| < 0.4$. Then, elementary Euclidean geometry gives us
    \begin{equation*}
        \sin \theta \leq \frac{\|E'(V)\|}{\|d \Theta|_z(V)\|} \leq \frac{0.4}{1.5} < \frac{\sqrt{3}}{2}.
    \end{equation*}
    This shows that $\theta < \frac{\pi}{3}$. 
\end{proof}

\begin{remark} \label{rmk:bounding singular values}
    We explain how we find that the singular values of $M$ are greater than $1.5$. Since the entries of $M^t M$ are obtained from taking sums and products of entries of $M$, we may compute them precisely using floating-point arithmetic. This further implies that we may compute precisely the coefficients of the characteristic polynomial $\chi$. We find that the characteristic polynomial has 10 distinct roots $x_i$ by finding $a_i, b_i \in \RR$ such that $a_i < b_i$ and $\chi(a_i) \chi(b_i) < 0$. The intermediate value theorem then guarantees the existence of the root $x_i \in (a_i, b_i)$ for each of the ten pairs $(a_i, b_i)$. We omit the tedious computation. 
\end{remark}

\subsection{Bounds on Angles and Hyperbolic Distances} \label{subsec:crude bounds}

Before proving Lemma \ref{lem:bounding second order partial}, we need to first establish some bounds on the angles and hyperbolic distances of the embedded surfaces in $\cB$. 

\begin{lemma} \label{lem:Euclidean and Hyperbolic Klein norm difference}
    Consider the closed ball $B_r(0) \subset \RR^3$ of radius $r \in (0,1)$ centered at the origin. Then, for all $X \in B_r(0)$, $V \in \RR^3$, we have 
    \begin{equation} \label{eqn:Euclidean and Hyperbolic Klein norm difference}
        \| V \|^2 \leq \| V \|_X^2 \leq \frac{\| V \|^2}{(1 - r^2)^2} .
    \end{equation}
\end{lemma}

\begin{proof}
    The first inequality follows from monotonicity of $1 / (1 + x^2)$ n $[0,r]$:
    \begin{equation*}
        \langle V, V \rangle_X 
        = \frac{\langle V, V \rangle}{a_X} + \frac{\langle V, X \rangle^2}{a_X^2}
        \geq \frac{\| V \|^2}{a_X}
        \geq \| V \|^2.
    \end{equation*}
    The second inequality follows from a routine application of the Cauchy-Schwarz inequality and monotonicity of $1 / (1 - x^2)^2$ on $[0,r]$: 
    \begin{equation*}
        \langle V, V \rangle_X 
        \leq \frac{\| V \|^2}{a_X} + \frac{\| V \|^2 \| X \|^2}{a_X^2}
        = \frac{\| V \|^2}{a_X^2}
        \leq \frac{\| V \|^2}{(1 - r^2)^2}.
    \end{equation*}
    This completes the proof. 
\end{proof}

\begin{corollary} \label{cor:tangent vector norm bounds}
    Given $S \in \cB$ and an edge $e$ connecting $X$ to $Y$, let $V = Y - X$. Then, we have $\| V \|_X \in [0.5, 13]$.
\end{corollary}

\begin{proof}
    Denote by $\hat V = \hat Y - \hat X$, where $\hat V$ is an edge connecting $\hat X$ to $\hat Y$ in $\hat S$ that corresponds to $V$. We compute numerically using Equation \eqref{eqn:klein 3 dot product} that $\| \hat V \| \in [0.509, 1.561]$. Then, since $\hat S \subset B_{0.79}(0)$, we must have $S \subset B_{0.8}(0)$ for all $S \in \cB$, so applying Lemma \ref{lem:Euclidean and Hyperbolic Klein norm difference} with $r = 0.8$ gives us the desired bounds.
\end{proof}

\begin{proposition} \label{prop:hyperbolic triangle side length difference bound}
    Given $S \in \cB$, let $e$ be an arbitrary edge in $S$ connecting $X$ to $Y$ that corresponds to the edge $\hat e$ in $\hat S$ connecting $\hat X$ to $\hat Y$. Denote by $l(e)$ and $l(\hat e)$ the hyperbolic length of edges $e$ and $\hat e$ respectively. Then, we have 
    \begin{equation*}
        |l(e) - l(\hat e)| \leq 1.6 \times 10^{-17} .
    \end{equation*}
\end{proposition}

\begin{proof}
    From Corollary \ref{cor:tangent vector norm bounds} we know that $S \subset B_{0.8}(0)$. Also, we must have $\| X - \hat X \| \leq 10^{-18}$ and $\| Y - \hat Y \| \leq 10^{-18}$. To find bounds on the hyperbolic distance $d_g(X, \hat X)$ between $X$ and $\hat X$, recall that geodesics in $\HH^3$ are straight lines. We can therefore parametrize the geodesic from $X$ to $\hat X$ as $\gamma(t) = (1-t) X + t \hat X$, where $t \in [0,1]$. It follows that 
    \begin{equation*}
        d_g(X, \hat X)
        = l(\gamma(t))
        = \int_0^1 \| \gamma' \|_{\gamma(t)} dt
        = \int_0^1 \| \hat X - X \|_{\gamma(t)} dt
        \leq 8 \| \hat X - X \| 
        \leq 8 \times 10^{-18}
    \end{equation*}
    where the second to the last inequality comes from Lemma \ref{lem:Euclidean and Hyperbolic Klein norm difference} and $1 / (1 - 0.8^2)^2 \leq 8$. We similarly get $d_g(Y, \hat Y) \leq 8 \times 10^{-18}$. It follows that
    \begin{equation*}
        \begin{aligned}  
        l(e) 
        &= d_g(X, Y)
        \leq d_g(X, \hat X) + d_g(\hat X, \hat Y) + d_g(\hat Y, Y)
        \leq 1.6 \times 10^{-17} + l(\hat e) \\
        l(\hat e)
        &= d_g(\hat X, \hat Y)
        \leq d_g(\hat X, X) + d_g(X, Y) + d_g(Y, \hat Y)
        \leq 1.6 \times 10^{-17} + l(e)
        \end{aligned}
    \end{equation*}
    Combining these two equations gives us the desired result.
\end{proof}

\begin{corollary} \label{cor:hyprbolic triangle side length bounds}
    With the same setup as above, we have $l(e) \in [0.6, 2.1]$. 
\end{corollary}

\begin{proof}
    Using the coordinates of $\hat S$, we compute that $l(\hat e) \in [0.63, 2.08]$. We explain how we establish the bound. 
    Given two distinct points $X, Y \in \HH^3$, the quadratic polynomial $\| (1 - t)X + tY \|^2 = 1$ has two distinct real roots $t_1 < 0$ and $t_2 > 1$. Let $A = (1 - t_1)X + t_1Y$ and $B = (1 - t_2)X + t_2Y$. The hyperbolic distance $d_g(X, Y)$ is given by 
    \begin{equation*}
        d_g(X, Y) 
        = \frac{1}{2} \ln \left( \frac{\| X - B \| \| Y - A \|}{\| X - A \| \| Y - B \|} \right)
        = \frac{1}{2} \ln \left( \frac{t_2 - t_1t_2}{t_1 - t_1t_2} \right).
    \end{equation*}

    We may express $\| (1 - t)X + tY \|^2 = 1$ as $at^2 + bt + c = 0$, where 
    \begin{equation*}
        a = \langle Y - X, Y - X \rangle ; \hspace{4pt}
        b = 2 \langle X, Y - X \rangle ; \hspace{4pt}
        c = \langle X, X \rangle - 1 .
    \end{equation*}
    Denote by $\Delta = b^2 - 4ac$ the discriminant of the quadratic $at^2 + bt + c = 0$. One can check that $\Delta > 0$ if $X \neq Y$. It follows that 
    \begin{equation} \label{eqn:hyprbolic triangle side length bounds}
        d_g(X, Y) 
        = \frac{1}{2} \ln \left( \frac{\Delta^{1/2} - b - 2c}{-\Delta^{1/2} - b - 2c} \right) 
        = \ln \left( \Delta^{1/2} - b - 2c \right) 
        - \frac{1}{2} \ln \left( 4c^2 + 4ac + 4bc \right)
    \end{equation}
    Since the square root function is monotonic increasing, it suffices to find floating-point numbers $x_1, x_2$ such that $x_1^2 \leq x \leq x_2^2$ to show $\sqrt{x} \in [x_1,x_2]$. This allows us to bound the value of $\Delta^{1/2}$ within $10^{-32}$. It follows that $\Delta^{1/2} - b - 2c \in [1.93, 6.3]$. We also directly computed that $4c^2 + 4ac + 4bc \in [0.62, 1.99]$. It remains to compute $\ln(x)$ for $x \in [0.62, 6.3]$, which gives us $l(\hat e) \in [0.63, 2.08]$. We refer to $\S$\ref{subsec:transcendental functions} for how we obtained this bound. 
    The corollary then follows from Proposition \ref{prop:hyperbolic triangle side length difference bound}. 
\end{proof}

\begin{proposition} \label{prop:hyperbolic law of cosine lipschitz}
    The map 
    \begin{equation} \label{eqn:hyperbolic law of cosine lipschitz}
        \psi(a, b, c) = \frac{\cosh a \cosh b - \cosh c}{\sinh a \sinh b}
    \end{equation}
    is $70$-Lipschitz on $[0.6, 2.1]^3$. 
\end{proposition}

\begin{proof}
    The partial derivatives of $\psi$ are given by 
    \begin{equation*}
        \pdv{\psi}{a} = \frac{\cosh a \cosh c - \cosh b}{\sinh^2 a \sinh b} ; \hspace{10pt}
        \pdv{\psi}{b} = \frac{\cosh b \cosh c - \cosh a}{\sinh^2 b \sinh a} ;  \hspace{10pt}
        \pdv{\psi}{c} = - \frac{\sinh c}{\sinh a \sinh b}.
    \end{equation*}
    Since $\sinh, \cosh, \tanh$ are all positive and monotonic increasing on $[0.5, 2.1]$, we have $|\pdv{\psi}{c}| \leq \frac{\sinh 2.1}{\sinh^2 0.5} \leq 15$ and 
    \begin{equation} \label{eqn:hyperbolic law of cosine lipschitz proof}
        \left| \pdv{\psi}{a} \right|
        \leq \left| \frac{\cosh c}{\tanh a \sinh a \sinh b} \right| + \left| \frac{1}{\sinh^2 a \tanh b} \right|
        \leq \frac{\cosh 2.1 + 1}{\tanh 0.5 \sinh^2 0.5}
        \leq 42.
    \end{equation}
    See $\S$\ref{subsec:transcendental functions} for finding bounds on the values of $\sinh x, \cosh x, \tanh x$ where $x = 0.5, 2.1$.
    Similarly, $| \pdv{\psi}{b} | \leq 42$. This implies 
    \begin{equation*}
        \frac{|\psi(y) - \psi(x)|}{\| y - x\|}
        \leq \sup_{I} \sqrt{\left( \pdv{\psi}{a} \right)^2 + \left( \pdv{\psi}{b} \right)^2 + \left( \pdv{\psi}{c} \right)^2} 
        \leq \sqrt{42^2 + 42^2 + 15^2}
        \leq 70
    \end{equation*}
    for all $x,y \in [0.6, 2.1]^3$. 
\end{proof}

\begin{corollary} \label{cor:sine angle bounds}
    For all $S \in \cB$ and an angle $\theta$ of a triangle in $S$, we have $|\sin \theta| \geq 0.24$. 
\end{corollary}

\begin{proof}
    Let $\hat\theta$ be the corresponding angle of $\theta$ in $\hat S$. 
    We compute using Equation \eqref{eqn:hyperbolic Klein dot product} and \eqref{eqn:hyperbolic Klein angle} that $\hat \theta$ satisfies $\cos \hat\theta \in [-0.008, 0.96]$. Denote by $a, b, c$ the length of the edges of the triangle that contains $\theta$, where $c$ corresponds to the opposite edge of $\theta$. Denote by $\hat a, \hat b, \hat c$ the edge lengths in $\hat S$. 
    From Proposition \ref{prop:hyperbolic triangle side length difference bound}, we know that $|a - \hat a| \leq 1.6 \times 10^{-17}$. Similarly, $|b - \hat b| \leq 1.6 \times 10^{-17}$ and $|c - \hat c| \leq 1.6 \times 10^{-17}$. It follows that $|\sqrt{a^2 + b^2 + c^2} - \sqrt{\hat a^2 + \hat b^2 + \hat c^2}| \leq 2.8 \times 10^{-17}$. Also, the hyperbolic law of cosine implies 
    \begin{equation*}
        \cos \theta = \frac{\cosh a \cosh b - \cosh c}{\sinh a \sinh b} = \psi(a, b, c).
    \end{equation*}
    From Proposition \ref{prop:hyperbolic law of cosine lipschitz} we know that $\psi$ is $70$-Lipschitz on $[0.6, 2.1]^3$. Also, Proposition \ref{cor:hyprbolic triangle side length bounds} gives us $a,b,c,\hat a, \hat b, \hat c \in [0.6, 2.1]$, so $|\cos\theta - \cos\hat\theta \, | \leq 2 \times 10^{-15}$. This gives us $\cos\theta \in [-0.01, 0.961]$. It remains a routine computation to check that $|\sin\theta| \geq 0.24$.
\end{proof}

\subsection{Bounds on the Second-Order Partials} \label{subsec:second order partial}

Given $S \in \cB$, for a triangle $(i, j, k) \in F$, recall our notation $V = X_j - X_i$, $W = X_k - X_i$, $z_i$ the $z$-coordinate of $X_i$, and $\theta_{ijk}$ the angle at $X_i$. Throughout this section, we write 
\begin{equation*}
    u = \langle V, W \rangle_{X_i} ; \hspace{5pt}
    v = \| V \|_{X_i} ; \hspace{5pt}
    w = \| W \|_{X_i} ; \hspace{5pt}
    \theta = \theta_{ijk} = \arccos(u(vw)^{-1}).
\end{equation*}
We remind the readers of the bounds we obtained from $\S$\ref{subsec:crude bounds}:
\begin{itemize}
    \item $\| X_i \| \leq 0.8$ ($S \subset B_{0.8}(0)$);
    \item $\| V \|, \| W \| \leq 1.6$ ($S \subset B_{0.8}(0)$); 
    \item $v, w \in [0.5, 13]$ (Corollary \ref{cor:tangent vector norm bounds});
    \item $|\sin \theta| \geq 0.24$ (Corollary \ref{cor:sine angle bounds});
    \item $|z_i| \leq 0.8$ ($S \subset B_{0.8}(0)$). 
\end{itemize}
We remark that the Cauchy-Schwarz inequality further implies $|u| \leq 1.6^2$. 

\begin{lemma} \label{lem:bounds for first order partial}
    The following bounds on the first-order partials hold for all $l = i, j, k$:
    \begin{enumerate}
        \item $|\partial_l u| \leq 10^3$;
        \item $|\partial_l v|, |\partial_l w| \leq 10^3$; 
        \item $|\partial_l \theta| \leq 10^{7}$.
    \end{enumerate}
\end{lemma}

\begin{proof}
    To see $|\partial_l u| \leq 10^3$, Equation \eqref{eqn:partial i dot VW} and \eqref{eqn:partial j dot VW} imply 
    \begin{equation*}
        |\partial_i u|
        \leq \frac{4 \cdot 0.8}{1 - 0.8^2} + \frac{8 \cdot 0.8 \cdot 1.6^2}{0.36^2} + \frac{4 \cdot 0.8 \cdot 1.6^4}{0.36^3}
        \leq 10^3 ; \hspace{10pt}
        |\partial_l u|
        \leq \frac{2 \cdot 0.8}{0.36} + \frac{0.8 \cdot 1.6^2}{0.36^2}
        \leq 10^3
    \end{equation*}
    for $l = j$ and $k$. Here, the denominator $1 - \langle X_i, X_i \rangle$ is bounded below by $1 - 0.8^2 = 0.36$. 

    To see $|\partial_l v| \leq 10^3$, we first observe that $(2v)^{-1} \leq (2 \cdot 0.5)^{-1} = 1$. Then, Equation \eqref{eqn:partial i norm V} and \eqref{eqn:partial j norm V} imply 
    \begin{equation*}
        |\partial_i v| 
        \leq \frac{4 \cdot 0.8}{0.36} + \frac{8 \cdot 0.8 \cdot 1.6^2}{0.36^2} + \frac{4 \cdot 0.8 \cdot 1.6^4}{0.36^3} 
        \leq 10^3 ; \hspace{10pt}
        |\partial_j v|
        \leq \frac{4 \cdot 0.8}{0.36} + \frac{2 \cdot 0.8 \cdot 1.6^2}{0.36^2} 
        \leq 10^3.
    \end{equation*}
    Moreover, we have $\partial_k v = 0$. This shows $|\partial_l v| \leq 10^3$ for all $l$. The argument for $|\partial_l w| \leq 10^3$ is symmetric, so we will omit it. 

    Finally, to see $|\partial_l \theta| \leq 10^{7}$, Equation \eqref{eqn:partial angle} implies 
    \begin{equation*}
        |\partial_l \theta|
        = \frac{\partial_l (vw) \cos \theta - \partial_l u}{vw \sin \theta}
        \leq \frac{2 \cdot 1.6 \cdot 10^3 + 10^3}{0.5^2 \cdot 0.24}
        \leq 3.2 \cdot 10^6 \leq 10^7.
    \end{equation*}
    This completes the proof.
\end{proof}

\begin{lemma}
    The following bounds on the second-order partials hold for all $l, m \in \{i, j, k\}$:
    \begin{enumerate}
        \item $|\partial_{lm} u| \leq 10^{4}$;
        \item $|\partial_{lm} v|, |\partial_{lm} w| \leq 10^{7}$.
    \end{enumerate}
\end{lemma}

\begin{proof}
    To ease notation we write
    \begin{equation*}
        a = 1 - \| X_i \|^2 ; \hspace{5pt}
        t_1 = \langle X_i, V \rangle ; \hspace{5pt}
        t_2 = \langle X_i, W \rangle ; \hspace{5pt}
        t_3 = \langle V, W \rangle ; \hspace{5pt}
        t_4 = \| V \|^2 .
    \end{equation*}
    From Cauchy-Schwarz inequality, all the $t_i$'s satisfy $|t_i| \leq 1.6^2$, and $a > 1 - 0.8^2 = 0.36$.

    We start by computing the second-order partials. 
    Equation \eqref{eqn:partial i dot VW} and \eqref{eqn:partial j dot VW} gives us 
    \begin{equation*}
        \begin{aligned}
            \partial_{ii} u
            =& \frac{2}{a} + \frac{2 (t_3 - t_1 - t_2 + 8 z_i^2 - 4 z_i z_j - 4 z_i z_k + z_j z_k)}{a^2} \\
            &+ \frac{4(t_1 t_2 + 2 t_1 z_iz_k - 4t_1z_i^2 + 2 t_2 z_iz_j - 4t_2z_i^2 + 2t_3 z_i^2)}{a^3} + \frac{24 t_1 t_2 z_i^2}{a^4} ;
        \end{aligned}
    \end{equation*}
    \begin{equation*}
        \partial_{ij} u
        = \frac{-1}{a} + \frac{t_2 - 4z_i^2 + 3z_iz_k}{a^2} + \frac{4 t_2 z_i^2}{a^3} ; \hspace{15pt}
        \partial_{jk} u
        = \frac{1}{a} + \frac{z_i^2}{a^2} ;
    \end{equation*}
    and $\partial_{jj} u = \partial_{kk} u = 0$. $\partial_{ik} u$ follows from symmetry. 

    For the second-order partials of $v$, we have 
    \begin{equation*}
        \partial_{ii} v
        = \frac{1}{av} + \frac{t_4 - 2t_1 + 8z_i^2 - 8z_iz_j + z_j^2}{a^2v} + \frac{2 (2t_4z_i^2 + 4t_1z_iz_j - 8t_1z_i^2 + t_1^2)}{a^3v} + \frac{12 t_1^2 z_i^2}{a^4v} - \frac{(\partial_i v)^2}{v};
    \end{equation*}
    \begin{equation*}
        \partial_{ij} v
        = \frac{-1}{av} + \frac{3z_iz_j - 4z_i^2 + t_1}{a^2v} + \frac{4t_1z_i^2}{a^3v}- \frac{\partial_i v \partial_j v}{v} ; 
        \hspace{15pt}
        \partial_{jj} v
        = \frac{- (\partial_j v)^2}{v} + \frac{1}{av} + \frac{z_i^2}{a^2v} ;
    \end{equation*}
    and all other partials vanish. 

    Next, we find bounds on all these partials. 
    The partials $|\partial_{lm} u|$ are bounded above by $|\partial_{ii} u|$, where 
    \begin{equation*}
        |\partial_{ii} u|
        \leq \frac{2}{0.36} + \frac{6 \cdot 1.6^2 + 34 \cdot 0.8^2}{0.36^2} + \frac{4 \cdot 1.6^4 + 56 \cdot 1.6^2 \cdot 0.8^2}{0.36^3} + \frac{24 \cdot 1.6^4 \cdot 0.8^2}{0.36^4} 
        \leq 10^4.
    \end{equation*}
    The partials $|\partial_{lm} v|$ are bounded above by $\partial_{ii} v$, where
    \begin{equation*}
        |\partial_{ii} v|
        \leq \frac{1}{0.36 \cdot 0.5} + \frac{3 \cdot 1.6^2 + 17 \cdot 0.8^2}{0.36^2 \cdot 0.5} + \frac{24 \cdot 1.6^2 \cdot 0.8^2 + 2 \cdot 1.6^4}{0.36^3 \cdot 0.5} + \frac{12 \cdot 1.6^4 \cdot 0.8^2}{0.36^4 \cdot 0.5} + \frac{10^6}{0.5}
        \leq 10^7.
    \end{equation*}
    The argument for $|\partial_{lm} w| \leq 10^7$ is symmetric, so we will omit it.
\end{proof}

\begin{proof}[Proof of Lemma \ref{lem:bounding second order partial}]
    From Equation \eqref{eqn:partial angle}, we compute $\partial_{lm} \theta$ as 
    \begin{equation*}
        \begin{aligned}
            \partial_{lm} \theta 
            &= \frac{\partial_{lm} (vw) \cos \theta - \partial_l (vw) \sin\theta \, \partial_m \theta - \partial_{lm} u}{(v^2w^2 - u^2)^{1/2}}
            - \frac{(\partial_l (vw) \cos \theta - \partial_l u) \, \partial_l (v^2w^2 - u^2)}{2(v^2w^2 - u^2)^{3/2}} \\
            &= \frac{\partial_{lm} (vw) \cos \theta - \partial_l (vw) \sin\theta \, \partial_m \theta - \partial_{lm} u}{vw \sin \theta}
            - \frac{(\partial_l (vw) \cos \theta - \partial_l u) \, \partial_l (v^2w^2 - u^2)}{2 v^3 w^3 \sin^3 \theta}.
        \end{aligned}
    \end{equation*}
    Then, from $|\sin\theta \geq 0.24|$, $u \leq 1.6^2$, and $v,w \in [0.5, 13]$ we know that $|\partial_{lm} \theta|$ is bounded above by 
    \begin{equation*}
        \frac{2 \cdot 10^6 + 2 \cdot 13 \cdot 10^7 + 2 \cdot 13 \cdot 10^{10} + 10^4}{0.5^2 \cdot 0.24}
        + \frac{(2 \cdot 13 \cdot 10^3 + 10^3) \cdot (4 \cdot 13 \cdot 10^3 + 2 \cdot 13^2 \cdot 10^3)}{2 \cdot 0.5^6 \cdot 0.24^3} ,
    \end{equation*}
    which is less than $10^{14}$.
    This completes the proof.
\end{proof}

\printbibliography

\section{Appendix} \label{sec:appendix}

\subsection{Computing the Exponential Function} \label{subsec:transcendental functions}

In this section, we explain how we obtain bounds on the values of the exponential function, a transcendental function, using floating-point arithmetic. Let $S_n(x)$ be the $n$-th partial sum of the Taylor expansion of $e^x$ centered at 0. That is, 
\begin{equation*}
    S_n(x) = \sum_{k=0}^n \frac{x^k}{k!} ; \hspace{10pt}
    e^x = \sum_{k=0}^\infty \frac{x^k}{k!}
    = \lim_{n \rightarrow \infty} S_n(x).
\end{equation*}

Write $f(x) = e^x$. 
Given $a \in \ZZ_{>0}$, on the interval $[-a, a]$ we have
\begin{equation*}
    |f^{(n+1)}(x)| 
    = |e^x| 
    = e^x 
    \leq e^a
    \leq 3^a .
\end{equation*}
It follows that for all $x \in [-a, a]$, we have 
\begin{equation} \label{eqn:bounds for exp taylor series error}
    |S_n(x) - e^x| \leq \frac{a^{n+1} 3^{a}}{(n+1)!} .
\end{equation}

We are now ready to explain the computer calculations from Corollary \ref{cor:hyprbolic triangle side length bounds}. The task is to find an interval $[a, b]$ such that $a \leq \ln x \leq b$ for some given $x \in [0.6, 6.3]$. By monotonicity of the natural log function, it suffices to check that $e^a \leq x \leq e^b$. In particular, we noticed that $e^{-2} \leq 2^{-2} = 0.25 \leq 0.6$ and $e^2 \geq 2.7^2 \geq 7.2 \geq 6.3$. This implies for all $x$ such that $e^x \in [0.6, 6.3]$, we have $x \in [-2, 2]$. Then, using Equation \eqref{eqn:bounds for exp taylor series error}, we computed with $10^4$-digit precision in mpmath that $|S_{20}(x) - e^x| \leq 10^{-10}$ for all $x \in [-2, 2]$. This justifies our calculation in Corollary \ref{cor:hyprbolic triangle side length bounds}.

In Proposition \ref{prop:hyperbolic law of cosine lipschitz} we computed $\sinh x, \cosh x, \tanh x$ for $x = 0.5$ and $x = 2.1$. Using Equation \eqref{eqn:bounds for exp taylor series error}, we see that $|S_{20}(x) - e^x| \leq 10^{-8}$ for $x \in [-3, 3]$. Denote by 
\begin{equation*}
    S(x) = \frac{S_{20}(x) - S_{20}(-x)}{2} ; \hspace{5pt} 
    C(x) = \frac{S_{20}(x) + S_{20}(-x)}{2} ; \hspace{5pt} 
    T(x) = \frac{S_{20}(x) - S_{20}(-x)}{S_{20}(x) + S_{20}(-x)} .
\end{equation*}
It follows that $|S(x) - \sinh x| \leq 10^{-8}$ and $|C(x) - \cosh x| \leq 10^{-8}$ for all $x \in [-3, 3]$. 
Also, observe that $|\sinh x| \leq \cosh x \leq 3^3 = 27$ for $x \in [-3, 3]$, and $\cosh x \geq 1$ by AM-GM inequality, we have that 
\begin{equation*}
    \begin{aligned}
        \left| T(x) - \tanh x \right|
        &= \left| \frac{S(x)}{C(x)} - \frac{\sinh x}{\cosh x} \right| 
        \leq \frac{|\cosh x| |S(x) - \sinh x| + |\sinh x| |C(x) - \cosh x|}{|C(x)| |\cosh x|} \\
        &\leq \frac{2 \cdot 27 \cdot 10^{-8}}{0.9} \leq 10^{-6}.
    \end{aligned}
\end{equation*}

We then computed $S(x)$, $C(x)$ and $T(x)$ for $x = 0.5$ and $x = 2.1$ to get 
\begin{equation*}
    \begin{matrix}
        \sinh 0.5 \in [0.521, 0.522] & \cosh 0.5 \in [1.127, 1.128] & \tanh 0.5 \in [0.462, 0.463] \\
        \sinh 2.1 \in [4.021, 4.022] & \cosh 2.1 \in [4.144, 4.145] & \tanh 2.1 \in [0.970, 0.971]
    \end{matrix}
\end{equation*}

\subsection{The Coordinates for Robust Flatness} \label{subsec:flatness coords}

Here we provide the coordinates $Y_{i, n_j}$ as promised in the proof of Proposition \ref{prop:flatness of S hat}. We do so vertex-by-vertex.

For vertex 0, we have $d_0 = 9$ and $(n_0, \ldots, n_8) = (1, 7, 8, 5, 3, 6, 4, 9, 2)$:
\begin{equation*}
    \begin{matrix}
        Y_{0,1} & 69773419400785390350719255734910 & 0 \\
        Y_{0,7} & 22673601164050404525981471083113 & 30522649530701644390945437817321 \\
        Y_{0,8} & -4502266020035690951698303492671 & 49787214971963367839502497819870 \\
        Y_{0,5} & -27890676491289921725823658542527 & 45152905884172442060397174220700 \\
        Y_{0,3} & -73673356961053011446423416051151 & 12554263676803953556985351161533 \\
        Y_{0,6} & -65716941293368369462892421315788 & -20094064505362122337999561399511 \\
        Y_{0,4} & -63585776023578037908056362846628 & -44491862575253747132630394705349 \\
        Y_{0,9} & -20681721190011172610982530346788 & -48616891910058401869453205847887 \\
        Y_{0,2} & 34828624619056461285801375727185 & -44302110482568969803994230453983 \\
    \end{matrix}
\end{equation*}

For vertex 1, we have $d_1 = 8$ and $(n_0, \ldots, n_7) = (0, 2, 4, 6, 5, 8, 3, 7)$:
\begin{equation*}
    \begin{matrix}
        Y_{1,0} & 69773419400785390350719255734910 & 0 \\
        Y_{1,2} & 60048146452422302324334406903101 & 34003898428632477407338918008929 \\
        Y_{1,4} & 684253873310370060705932765911 & 45011277063511932691880974039935 \\
        Y_{1,6} & -38797441527227489121142967415251 & 35904267659650216159151157407406 \\
        Y_{1,5} & -34418586117854334955122669640154 & -13661872610266486281752985853784 \\
        Y_{1,8} & -13188192204919717548387089195338 & -32334046201585819117670137142096 \\
        Y_{1,3} & 27532806238161586916092515324102 & -59644150643141823546483370604325 \\
        Y_{1,7} & 61207051193049841346494694029781 & -20774055887376880717112213475032
    \end{matrix}
\end{equation*}

For vertex 2, we have $d_2 = 8$ and $(n_0, \ldots, n_7) = (0, 9, 6, 3, 8, 7, 4, 1)$:
\begin{equation*}
    \begin{matrix}
        Y_{2,0} & 56353438990578855635490869491526 & 0 \\
        Y_{2,9} & 48092387991317990941147748740637 & 39405069223657519928257610119072 \\
        Y_{2,6} & 36264020972833155000907130948793 & 56108711187827723236898114831237 \\
        Y_{2,3} & -8023794381899623805048972165114 & 49766168048906812726617349210993 \\
        Y_{2,8} & -46688768526075731314833760652532 & 14579781429380556985176995990071 \\
        Y_{2,7} & -34323173836899021622431417499987 & -16589459188636414317131134542419 \\
        Y_{2,4} & 15801788573781581976615289172374 & -66413096470080882136234194648455 \\
        Y_{2,1} & 40340533719441891677134752026234 & -55988269663804427444966060569426
    \end{matrix}
\end{equation*}

For vertex 3, we have $d_3 = 8$ and $(n_0, \ldots, n_7) = (0, 5, 4, 7, 1, 8, 2, 6)$:
\begin{equation*}
    \begin{matrix}
        Y_{3,0} & 74735353497374034571180561502097 & 0 \\
        Y_{3,5} & 64951355825035318887740051750196 & 27775730204660691498207963703796 \\
        Y_{3,4} & 18203731888080621301584871054039 & 44282741941529823958117353691764 \\
        Y_{3,7} & -50270269905279614476683928736812 & 43757175563606495111050506045891 \\
        Y_{3,1} & -65679523751189539111583804355939 & -1296258117380181104756639335080 \\
        Y_{3,8} & -53135278703958021717913038077077 & -20673081744056204160401703682562 \\
        Y_{3,2} & -21947094098998016793187902172829 & -45380368213190643026673636310429 \\
        Y_{3,6} & 41407235240572972447410876062927 & -39369780957805212274982571094628
    \end{matrix}
\end{equation*}

For vertex 4, we have $d_4 = 8$ and $(n_0, \ldots, n_7) = (0, 6, 1, 2, 7, 3, 5, 9)$:
\begin{equation*}
    \begin{matrix}
        Y_{4,0} & 77605906656232659918496595614600 & 0 \\
        Y_{4,6} & 35287460699942822202627864783646 & 31260480461897652641282558745742 \\
        Y_{4,1} & -12247171417163211623784470557145 & 43318472486111465943829115437047 \\
        Y_{4,2} & -66900524329443115864714765994097 & 13591017229050724080290741729896 \\
        Y_{4,7} & -71986590683672407001678934329849 & -6598661163810768934712510249647 \\
        Y_{4,3} & -28248252874446950359042214693411 & -38657124803256481078874090619402 \\
        Y_{4,5} & 39577284522237084988918330394325 & -48198246080382034917047601058065 \\
        Y_{4,9} & 58114222516063173451633323816533 & -23550433384261927623553718279566 
    \end{matrix}
\end{equation*}

For vertex 5, we have $d_5 = 7$ and $(n_0, \ldots, n_6) = (0, 8, 1, 6, 9, 4, 3)$:
\begin{equation*}
    \begin{matrix}
        Y_{5,0} & 53072353866459763807316258600612 & 0 \\
        Y_{5,8} & 15132621870398054709768798392140 & 26932978119784418547635926675570 \\
        Y_{5,1} & -17271886847456450664911538801150 & 32756186565288176500115728267363 \\
        Y_{5,6} & -53366672941583648447918823670646 & -264475383850852598262292837658 \\
        Y_{5,9} & -39121047935469870027816023468487 & -22605253051769535765860271234150 \\
        Y_{5,4} & 507268248785703193038086141638 & -62363250831743394667763519936132 \\
        Y_{5,3} & 30653902634212146206451278379057 & -63643601918847202789375850542438 
    \end{matrix}
\end{equation*}

For vertex 6, we have $d_6 = 7$ and $(n_0, \ldots, n_6) = (0, 3, 2, 9, 5, 1, 4)$:
\begin{equation*}
    \begin{matrix}
        Y_{6,0} & 68720359438100137565590539767661 & 0 \\
        Y_{6,3} & 25390557772908761618321852158345 & 51184551955783422998211990124352 \\
        Y_{6,2} & -6813962085088593529842703558335 & 66459285348099568242306767088465 \\
        Y_{6,9} & -29304204462392389437047998194621 & 15467804412244368520337457391781 \\
        Y_{6,5} & -49711247331825269118889114937918 & -19412975475096919280111640378696 \\
        Y_{6,1} & -32118833053426954804131807462710 & -41984979082054198551559535804433 \\
        Y_{6,4} & 5307487555958890405906055277685 & -46842855348516443461783670226668
    \end{matrix}
\end{equation*}

For vertex 7, we have $d_7 = 6$ and $(n_0, \ldots, n_5) = (0, 1, 3, 4, 2, 8)$:
\begin{equation*}
    \begin{matrix}
        Y_{7,0} & 38022681706061569845515613988920 & 0 \\
        Y_{7,1} & 11818625227981235011001127193086 & 63546712043039705538025314800530 \\
        Y_{7,3} & -37070064618057322788690894186616 & 55385925638802835554058530315113 \\
        Y_{7,4} & -60394084412598094355018887265961 & 39726139195463625691762018099811 \\
        Y_{7,2} & -35291354301407287983444968408251 & -14415641851196358579951401594919 \\
        Y_{7,8} & 6448073384057528034696596840928 & -37852482326636002546581327669442
    \end{matrix}
\end{equation*}

For vertex 8, we have $d_8 = 6$ and $(n_0, \ldots, n_5) = (0, 7, 2, 3, 1, 5)$:
\begin{equation*}
    \begin{matrix}
        Y_{8,0} & 49990370812584110646810728859197 & 0 \\
        Y_{8,7} & 28931837076682989366588153660884 & 25245531723906183515103808022266 \\
        Y_{8,2} & 5341182688010834270565922246268 & 48619778902245980212782746138780 \\
        Y_{8,3} & -36229090326845443500768801226589 & 44024847141506336954558235705473 \\
        Y_{8,1} & -26802550827068733629932236460907 & -22383972537567782704648869515785 \\
        Y_{8,5} & 7930519585113553825600252467951 & -29857803237736129523910125242990
    \end{matrix}
\end{equation*}

For vertex 9, we have $d_9 = 5$ and $(n_0, \ldots, n_4) = (0, 4, 5, 6, 2)$:
\begin{equation*}
    \begin{matrix}Y_{9,0} & 52833093515103383889224641548820 & 0 \\
        Y_{9,4} & 940902316584597264618732735964 & 62697691138220619249446404872648 \\
        Y_{9,5} & -38467987934475596300054436654094 & 23699530826092470056977361403538 \\
        Y_{9,6} & -20279822767601932923431973551822 & -26205307878821013050922465367016 \\
        Y_{9,2} & 43525294101239515382128354128724 & -44398040909811691945973843164387 
    \end{matrix}
\end{equation*}

\subsection{The Coordinates for the 12-Vertex Embedding} \label{subsec:12-vertex coords}

Here we provide the coordinates of the first six vertices of $\hat S_{12}$. One can obtain the last six vertices by applying the rotation $R$ about the $z$-axis by $\pi$. For $i = 0, 1, \ldots, 5$, let $V_i$ denote the image of vertex $i$ under $\tilde S_{12}$.
\begin{equation*}
    V_0 =
    \begin{pmatrix}
        -0.40125921389358304 \\
        -0.59971067001960654 \\
        -0.21300741678578042
    \end{pmatrix}
    \hspace{10pt}
    V_1 =
    \begin{pmatrix}
        -0.74004324087313211 \\
        -0.034941867295347241 \\
        -0.17336538136629426
    \end{pmatrix}
\end{equation*}
\begin{equation*}
    V_2 =
    \begin{pmatrix}
        -0.047326192857641246 \\
        0.73907490576107349 \\
        -0.16861249423379737
    \end{pmatrix}
    \hspace{10pt}
    V_3 =
    \begin{pmatrix}
        0.56818704554415911 \\
        -0.43330477745696788 \\
        -0.008615324240392469
    \end{pmatrix}
\end{equation*}
\begin{equation*}
    V_4 =
    \begin{pmatrix}
        -0.6211718336355635 \\
        -0.03799731280539783 \\
        -0.0054573602282373575
    \end{pmatrix}
    \hspace{10pt}
    V_5 =
    \begin{pmatrix}
        -0.49006353683388598 \\
        0.16430505278256752 \\
        0.37728234459977172
    \end{pmatrix}
\end{equation*}

\end{document}